\documentclass{article}

\usepackage{amssymb}
\usepackage{amsmath}
\usepackage{amsthm}
\usepackage{graphicx}
\usepackage{url}

\theoremstyle{definition}
\newtheorem{theorem}{Theorem}

\newtheorem{definition}{Definition}

\title{Graph Nimors}

\author{Matthew Skala\\mskala@ansuz.sooke.bc.ca}

\newcommand{\mex}{\mathop{\mathrm{mex}}}

\begin{document}

\maketitle


\begin{abstract}
In the game of \emph{Graph Nimors}, two players
alternately perform graph minor operations (deletion and contraction of
edges) on a graph until no edges remain, at which point the player who last
moved wins.  We present theoretical and experimental results and
conjectures regarding this game.
\end{abstract}


\section{Introduction}

\emph{Graph Nimors} is a combinatorial game in which two players take
alternating turns performing graph minor operations on a graph until no
edges remain, at which point the player who last moved wins.  Although the
rules of Graph Nimors are simple and seem obvious, it appears to be novel,
and its analysis appears to be difficult.  A simple online version, coded by
Martin Aum\"{u}ller, is online at {\url{http://itu.dk/people/mska/nimors/}}. 
How does one play this game well?

All our graphs are simple (no multiple edges) and undirected.  A \emph{graph
minor operation} consists of \emph{deleting} an edge or \emph{contracting}
an edge---that is, removing two adjacent vertices $u$ and $v$ and inserting
a new vertex $w$ adjacent to the union of the neighbourhoods of $u$ and $v$
in the original graph (excluding $u$ and $v$ themselves).  We will often
make use of the \emph{blocks} of a graph, which are maximal biconnected
subgraphs, allowing single edges as blocks, so that the edges of any graph
can be partitioned into a disjoint union of blocks.

Let $C_n$ denote the cycle of $n$ vertices and $n$ edges; $K_n$ denote
the complete graph on $n$ vertices (which has $n(n-1)/2$ edges); and
$K_{p,q}$ the complete bipartite graph with parts of $p$ and $q$
vertices (which has $p+q$ vertices and $pq$ edges).  The \emph{girth} of a
graph is the number of vertices in its smallest cycle, or $\infty$ if the
graph is acyclic.

Let $\oplus$ be the \emph{Nim sum}, a binary operator on
nonnegative integers usually described as ``binary addition without carry'';
it is also equivalent to bitwise exclusive OR, as in the C language
\verb|^| operator.  The ordinary sum is an upper bound on the Nim
sum.  Given a set $S$, let $\mex S$ be the least nonnegative
integer that is \emph{not} an element of $S$.  Note
$\mex\varnothing=0$ and $\mex S \le |S|$.  It is not necessary for all
elements of $S$ to be integers.  Other objects might occur when the theory is
extended to broader classes of games; but the $\mex$ of a set is by
definition a nonnegative integer.  Writing $\mex$ in Roman type is
the standard notation for this function, as used by other
authors~\cite{Berlekamp:Winning} and consistent with analogous functions
like $\max$ and $\min$.

A combinatorial game consists of a set of \emph{positions} with rules
describing, for any position $p$, sets of positions that are called
\emph{options} of $p$ for the Left player and for the Right player (thus, two
directed graphs).  If for all positions the Left and Right options are the
same, then the game is called \emph{impartial}; otherwise, \emph{partisan}. 
If for every position $p$, all directed walks starting from $p$ are of
finite length, then the game is \emph{short}.  All games we consider are
\emph{perfect-information} games, which means that each player knows the
current position (instead of only some function of it) when choosing which
move to make, and none involve random selections outside the players' control.

If, from a position in a combinatorial game, the next player to move can
win in all cases of the opponent's choices, then that is called an
$\mathcal{N}$-position (mnemonic: Next player to win).  If the other
player can win in all cases of the next player's choices, then that is
called a $\mathcal{P}$-position (Previous player to win).  In short impartial
games, every position is in one of these two classes; other kinds of games
admit other possibilities.  The \emph{standard play} convention is that
positions with no options, at which play necessarily terminates, are
$\mathcal{P}$-positions: a player unable to move loses.  The
\emph{mis\`{e}re play} convention is the opposite, with positions that have
no options defined to be $\mathcal{N}$-positions and a player unable to move
declared the winner.

The literature on combinatorial games is massive, and we survey only a few
of the most relevant results here.  The literature on graph minors is even
bigger; but as we use very little from that work here except for
starting from the idea of a ``graph minor operation,'' we will only refer
readers to the survey by Lovasz~\cite{Lovasz:Graph}.

The general theory of Nim-like games owes much to the theoretical work of
Sprague~\cite{Sprague:Mathematische} and Grundy~\cite{Grundy:Mathematics}
and the popular survey \emph{Winning Ways} of Berlekamp, Conway, and
Guy~\cite{Berlekamp:Winning}.  Graph Nimors as such appears to be novel, but
many other Nim-like games involving graphs are known.  \emph{Hackenbush},
which involves deleting subgraphs from a graph, is a constantly-used example
and reference point for putting values on other games in \emph{Winning
Ways}.

As Demaine~\cite{Demaine:Playing} describes, it is typical for short
two-player games to be PSPACE-complete.  Schaefer~\cite{Schaefer:Complexity}
shows PSPACE-completeness of several graph games including \emph{Geography},
in which players move a token from vertex to vertex of a directed graph,
never repeating an arc; and \emph{Node Kayles}, where a move is to claim a
vertex not adjacent to any already-claimed vertex (thus building an
independent set).  Fraenkel and Goldschmidt~\cite{Fraenkel:PSPACE} show
PSPACE-hardness for several more classes of games involving moving tokens
and marking vertices in graphs.  Bodlaender~\cite{Bodlaender:Complexity}
describes a graph colouring game in which players take turns colouring
vertices without giving any two adjacent vertices the same colour; the
number of colours needed for the first player to force a complete colouring
is a natural graph invariant related to this game.  Bodlaender shows that
the variant in which the order of colouring vertices is predetermined, is
PSPACE-complete, and gives partial results for variants without that
restriction.

Fukuyama~\cite{Fukuyama:Nim} describes \emph{Nim on graphs}, where each edge
of a graph contains a Nim pile and players take turns moving a token from
vertex to vertex, subtracting from the pile on each edge traversed.  If
every pile is of size 1 and the graph is made directed, this is the same as
Geography.  Calkin et al.~\cite{Calkin:Computing} describe \emph{Graph Nim},
in which a move consists of choosing one vertex and removing any nonempty
subset of the edges incident to it; in the case of paths, this is easily
seen to be equivalent to the take-and-break game Kayles~\cite[Chapter
4]{Berlekamp:Winning}.  Fraenkel and Scheinerman~\cite{Fraenkel:Deletion}
describe a deletion game on hypergraphs, with moves consisting of removing
vertices or hyperedges.  Harding and Ottaway~\cite{Harding:Edge} describe
edge-deletion games with constraints on the parity of the degrees of the
endpoints of the edges that may be deleted.  Henrich and
Johnson~\cite{Henrich:Link} describe a \emph{link smoothing} game, in which
players make ``smoothing'' moves on a planar embedding that represents the
shadow of a link diagram, attempting to either disconnect the diagram or
keep it connected.  Their work is of interest in the context of ours because
the smoothing moves are sometimes equivalent to edge contraction in a graph
representing the game state.  Few other games involving edge contraction are
known.


\section{Basic theory of Graph Nimors}

There is a general
theory~\cite{Sprague:Mathematische,Grundy:Mathematics,Berlekamp:Winning}
for a class of games that includes
Graph Nimors, summarized by the following well-known result.

\begin{theorem}[Sprague-Grundy Theorem]
For any short impartial two-player perfect-information combinatorial game
with the standard play convention and without randomness, there exists a
unique function $\mathcal{G}$ from positions to nonnegative integers,
called the \emph{Nim value} or \emph{Sprague-Grundy number}, with the
following properties where $G$ is any position of the game:
\begin{itemize}
\item If the options from $G$ are $G_1,G_2,$ $\ldots,G_k$, then
$\mathcal{G}(G)=\mex
\{\mathcal{G}(G_1),\mathcal{G}(G_2),$ $\ldots,\mathcal{G}(G_k)\}$.  This
implies $\mathcal{G}(G)=0$ if there are no options from $G$, because $\mex
\varnothing = 0$.
\item $\mathcal{G}(G)=0$ if and only if $G$ is a $\mathcal{P}$-position.
\item If $G$ can be separated into a union of two positions $G'$
and $G''$, such that each player's turn consists of moving in exactly one of
the sub-positions and where no
sequence of moves in one will affect the moves available in the other, then
$\mathcal{G}(G)=\mathcal{G}(G')\oplus\mathcal{G}(G'')$.
\item Optimal play is to move to any position of zero Nim value, which
is possible if and only if the current Nim value is nonzero.  Then the
opponent either loses immediately or is forced to move to a position of
nonzero Nim value, at which point one can apply the strategy again.
\end{itemize}
\end{theorem}

The prototype game meeting these conditions is \emph{Nim}:  a position is
some number of piles of stones, with the legal move being to remove any
nonempty subset of any one pile.  In that game the Nim value of a single
pile is simply the number of stones in it.  The Nim sum rule above is used
to evaluate multi-pile configurations, and that gives an easy winning
strategy.

The game of Graph Nimors also meets the conditions.  Blocks serve to
partition the graph.  No sequence of moves in one block can affect the moves
available in any other blocks.  Therefore the Nim value of a graph is the
Nim sum of the Nim values of its blocks.  Assuming we can easily find the
Nim values of biconnected graphs, we can compute them for any other
graphs, and thereby play optimally from any $\mathcal{N}$-position.

However, the only obvious way to compute the Nim value of a general
biconnected graph is to recursively examine all its minors, which is
prohibitively time-consuming in all but the smallest cases.

\subsection{Easy cases}\label{sub:easy-cases}

For very small graphs, the Nim value is easy to calculate by brute force. 
All biconnected graphs of up to four vertices are shown in
Figures~\ref{fig:4vert}, with arrows among graphs to
show the options from each position and a few extra graphs to illustrate
non-biconnected options for the four-vertex graphs.  Note that breaking
apart the blocks into separate components makes no difference to the Nim
value, and we do that in the figure to make the boundaries between blocks as
clear as possible.  The Nim value of each biconnected graph is the mex of
the Nim values for its options.  The biconnected graphs of five vertices and
their Nim values are shown in Figure~\ref{fig:5vert}, but even for graphs
as small as these, there are so many options that showing them all would
make the diagram excessively complicated.

\begin{figure}
\includegraphics[scale=0.78]{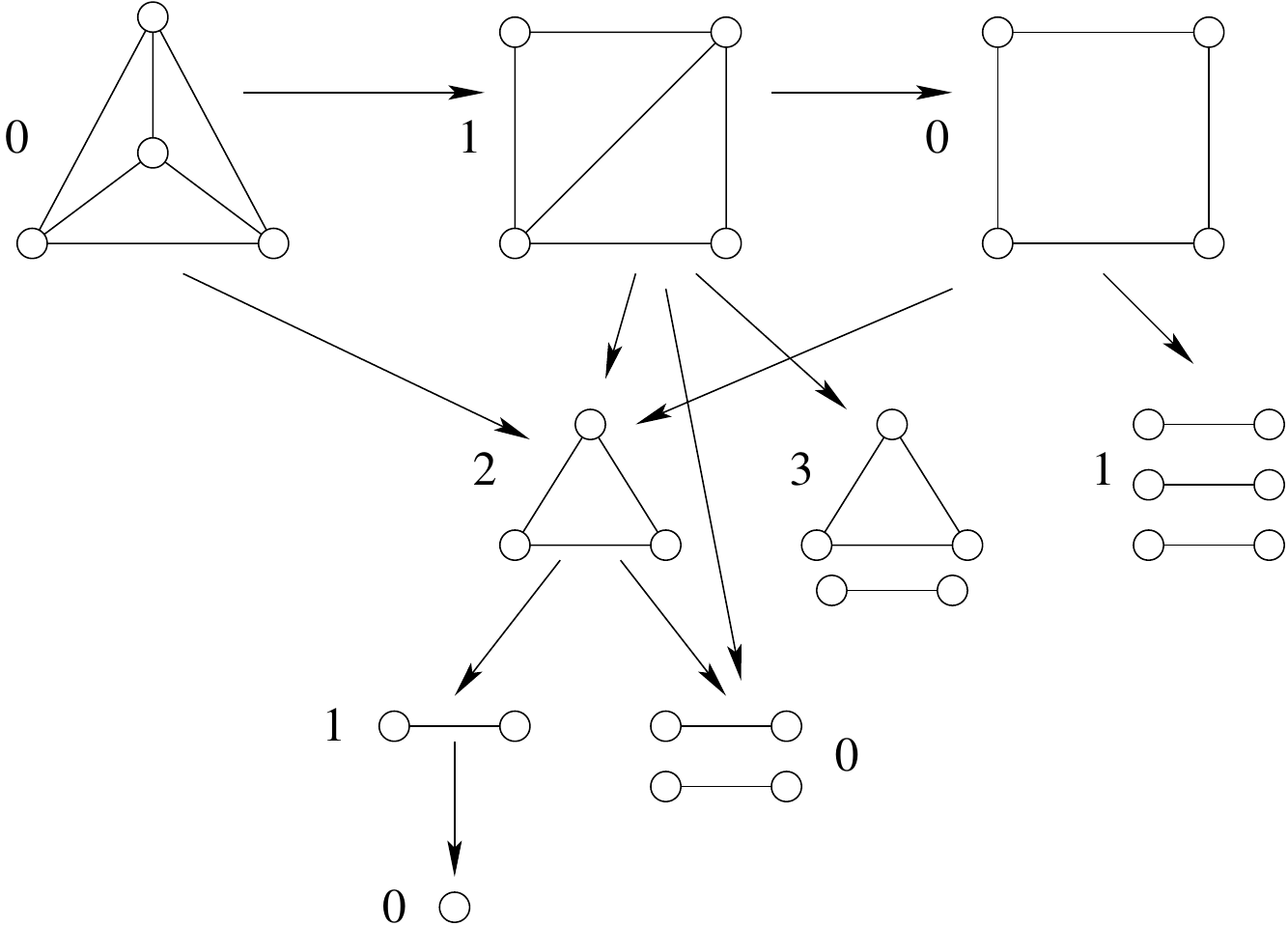}
\caption{The biconnected graphs of up to four vertices, and their Nim
values.}\label{fig:4vert}
\end{figure}

\begin{figure}
\includegraphics[scale=0.78]{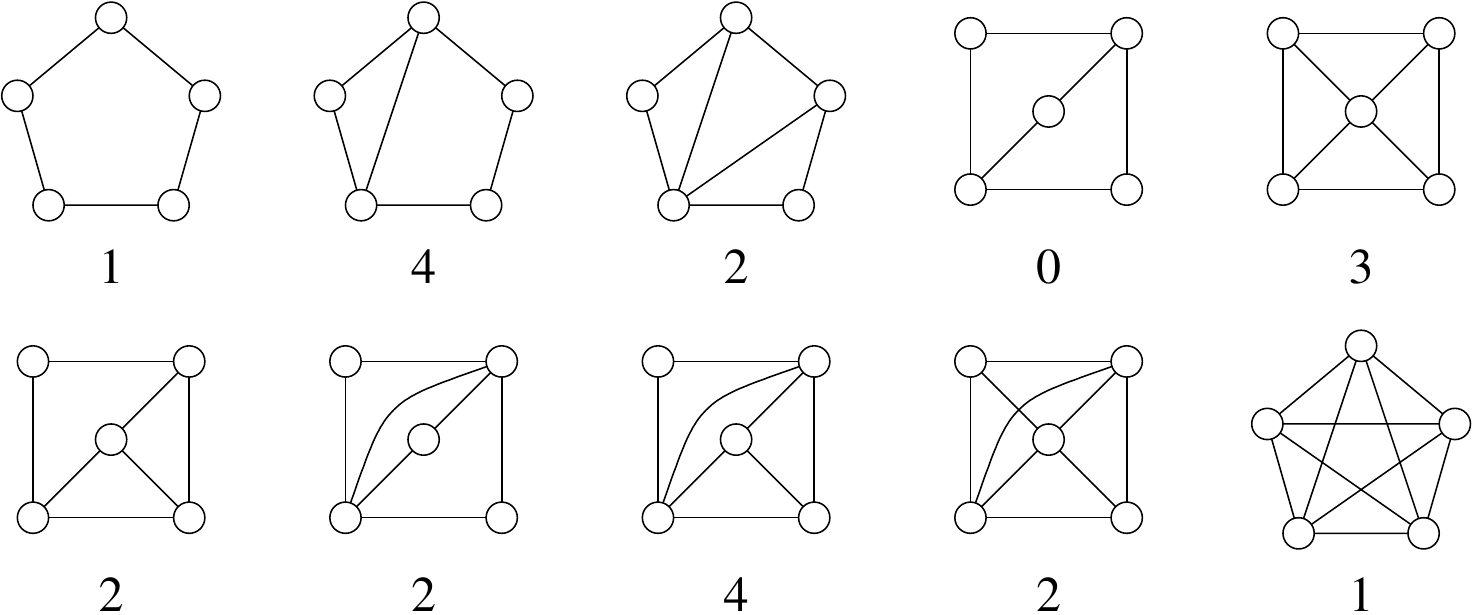}
\caption{The biconnected graphs of five vertices, and their Nim
values.}\label{fig:5vert}
\end{figure}

On an acyclic graph, every move reduces the edge count by exactly one.  The
game ends when the edges are exhausted, and the players' choices to delete
or contract edges make no difference to the final result.  The Nim value of
an acyclic graph is 0 if the number of edges is even, 1 if odd.  A graph
with no edges has Nim value 0 (no moves possible); with one edge, Nim value
1 ($\mex \{0\} = 1$), and then for larger acyclic graphs, each of the edges
is a block and we take the Nim sum of an even or odd number of them.

The Nim value of $C_3$ is 2, because its options are paths of
one and two edges, which have Nim values of 1 and 0, and $\mex \{0,1\}=2$. 
The Nim value of $C_4$ is 0 because its options are $C_3$ and a three-edge
path, and $\mex \{1,2\}=0$.  Larger cycles $C_k$ have Nim value 0 for even
$k$, 1 for odd $k$, by an easy induction.

Not many other cases can really be called ``easy.''  Even such a simple
thing as two cycles sharing one edge (equivalent to a cycle with a chord
across it) requires more than trivial work to analyse.

\begin{theorem}\label{thm:fused-cycle}
Let $FC_{p,q}$ (mnemonic:  ``fused cycle'')
be the graph of $p+q-2$ vertices and $p+q-1$ edges formed by
identifying one edge of $C_p$ with one edge of $C_q$.  Without loss of
generality assume $p\le q$.  Then
\begin{equation}
\begin{aligned}
  \mathcal{G}(FC_{3,3}) &=1 \\
  \mathcal{G}(FC_{3,4}) &=4 \\
  \mathcal{G}(FC_{3,q}) &=2 \text{ for odd } q\ge 5\\
  \mathcal{G}(FC_{3,q}) &=3 \text{ for even } q\ge 6\\
  \mathcal{G}(FC_{p,q}) &=0 \text{ for odd } p+q
    \text{ when } p\ge 4,\ q\ge 4\\
  \mathcal{G}(FC_{p,q}) &=1 \text{ for even } p+q
    \text{ when } p\ge 4,\ q\ge 4.\\
\end{aligned}
\end{equation}
\end{theorem}

\begin{proof}
For the case $FC_{3,3}$:  there are two kinds of edges, each of which may be
removed or contracted.  Removing the centre edge leaves $C_4$ with Nim value
0.  Removing a side edge leaves $C_3$ plus one edge, with Nim value $2
\oplus 1 = 3$.  Contracting the centre edge leaves two edges, with Nim value
$1 \oplus 1 = 0$.  Contracting a side edge leaves $C_3$ with Nim value 2. 
Then $\mex \{0,2,3\}=1$.

For the case $FC_{3,4}$:  We can delete or contract one of the edges that
came only from
$C_3$, from $C_4$, or the shared edge (six moves in all).  Deleting an edge
from $C_3$ leaves $C_4$ and one edge as blocks, total Nim value 1.  Deleting
an edge from $C_4$ leaves $C_3$ and two edges as blocks, total Nim value 2. 
Deleting the shared edge leaves $C_6$, Nim value 0.  Contracting an edge
from $C_3$ leaves $C_4$, Nim value 0.  Contracting an edge from $C_4$ leaves
$FC_{3,3}$, Nim value 1 (above).  Contracting the shared edge leaves $C_3$
plus an edge, Nim value 3.  Then $\mex \{0,1,2,3\}=4$.

For the case $FC_{3,q}$, $q\ge 5$:  Assume the theorem is true for smaller
$q$.  Deleting an edge from $C_3$ leaves $C_q$ plus an edge, Nim value 0 or
1 with the opposite parity from $q$.
Contracting an edge from $C_3$ leaves just $C_q$, Nim
value 0 or 1 with the \emph{same} parity as $q$.
Thus, these two cases together cover the Nim values 0
and 1.  Deleting or merging the shared edge leaves a cycle of length at
least $4$ and possibly an extra dangling edge; the Nim value of the result
is 0 or 1, and already covered.  Deleting an edge from $C_q$ leaves a
triangle and $q-2$ edges as blocks, with Nim value 2 for even $q$ and 3 for
odd $q$.  Merging an edge from $C_q$ leaves $FC_{3,q-1}$, which by the
inductive assumption has the same Nim value as $C_q$ plus an edge, namely 2
for even $q$ (odd $q-1$) and 3 for odd $q$ (even $q-1$), unless it is
$FC_{3,4}$ with Nim value 4.  Thus the values of the options are 0 and 1
unconditionally, exactly one of 2 or 3, and possibly also 4.
The mex of these values is 2 or 3, according to the parity of $q$:  2 for
odd $q$ and 3 for even $q$, and the result holds.

For the case $FC_{p,q}$ with $p\ge 4$ and $q \ge4$:  Assume the theorem is
true for smaller $p$ or $q$.  Deleting an edge from $C_p$ leaves as blocks
$C_q$ and $p-2$ single edges; the Nim value of the result is 0 or 1 with the
same parity as $p+q$.
Symmetrically, we get the same Nim value by deleting an edge from
$C_q$.  Deleting the shared edge leaves $C_{p+q-2}$, which also has the
same Nim value.
Contracting an edge in $C_p$ results in $FC_{p-1,q}$,
which by the inductive assumption has Nim value 0 or 1 with the same
parity as $p+q$, or else
greater than 1 (when $p=4$); and the same is true symmetrically of contracting an edge in
$C_q$.  That leaves only contracting the shared edge, which results in
$C_{p-1}$ and $C_{q-1}$ joined by a shared vertex, the Nim value of which
may be 0 or 1 with the same parity as $p+q$, or else (if exactly one
of $p$ and $q$ was equal to 4) a value greater than 1.  Thus the values of
the options are exactly one of the values
$\{0,1\}$ depending on the parity of $p+q$,
and possibly some value or values greater than 1.
The mex of this set is 0 or 1 with the opposite parity from $p+q$, and the
result holds.
\end{proof}

\subsection{Property S}

Girth seems relevant to the analysis of Graph Nimors, both because there are
some girth-related patterns visible in the computer results and because
there are simple statements we can make about the consequences of moves in
the game as they relate to girth.  A deletion move never decreases the
girth.  A contraction move never increases the girth, except in the special
case where it contracts an edge shared by all triangles in the graph, and if
it decreases the girth, it decreases the girth by exactly one.  Any move on
a graph of girth at least four (a triangle-free graph) subtracts exactly one
from the number of edges.

These facts suggest that if the starting girth is
sufficiently large, one player may be able to keep it large as part of a
simple winning strategy.  But actually implementing such a strategy seems
difficult.  For instance, the Petersen graph has girth 5 and Nim value 1. 
The first player, although able to win, cannot prevent the second player
from forming one or more triangles along the way.  The following property is
similar to girth, but represents something one player can preserve as part
of a strategy.

\begin{definition}
A graph $G$ has \emph{property S} (mnemonic: its high-degree
vertices are Separated by Series vertices) if it contains no edge incident
to two vertices of degree greater than two, and no block of $G$ is a
triangle.
\end{definition}

Note that property S implies $G$ is triangle-free.  The important
consequence of property S is that any move which reduces the girth can be
undone on the next move, allowing one player to force an outcome determined
by the parity of the number of edges.

\begin{theorem}\label{thm:property-s}
A graph with property S is an $\mathcal{N}$-position
if and only if it has an even number of edges.
\end{theorem}

\begin{proof}
Suppose $G$ has property S and an even number of edges.  If the first player
deletes an edge, then the result will have property S and an odd number of
edges, at which point the second player can delete any edge, preserving the
property and making the number of edges even again.  Similarly, if the first
player contracts an edge but leaves a graph that still has property S, then
the second player can delete any edge.

Suppose the first player contracts an edge in such a way that the resulting
graph does not have property S.  Then the first player's move must have
consisted of contracting an edge between a degree-two vertex and one of its
neighbours where both neighbours had degree greater than two, creating a new
edge between two vertices $u$ and $v$ of degree greater than two, as in
Figure~\ref{fig:uv-edge}.  The edge $(u,v)$ is the only one violating
property S.  Then the second player can delete that edge, restoring the
property and making the number of edges even.  By induction, the second
player has a winning strategy on any graph with property S and an even
number of edges.

On a graph with property S and an odd number of edges, the first player can
win by deleting any edge and then following the second-player strategy.
\end{proof}

\begin{figure}
\includegraphics[scale=0.8]{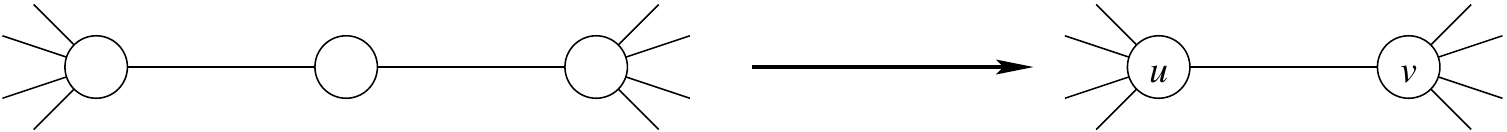}
\caption{Contracting an edge to a degree-two vertex.}\label{fig:uv-edge}
\end{figure}

Note that the winning strategies described in the proof only ever make use
of deletion moves, although the other player is free to contract edges.

\subsection{Bounds on the Nim value}

How large can the Nim value of a graph be?
The number of edges in the graph is an easy upper bound because the Nim
value of a graph can be at most the maximum Nim value of any option
of that graph, plus one.  Since every option of a graph $G$ has strictly
fewer edges than $G$, we can gain no more than one unit of Nim value for
each edge we add.  In fact, this bound is tight only for graphs of zero or
one edges; larger graphs always have Nim value strictly less than the number
of edges, because there is no two-edge graph of Nim value 2 and that
deficiency affects all larger graphs through the induction.

When a graph has some symmetry, there may be several edges which, if deleted
or contracted, give isomorphic results.  The number of options
\emph{distinct up to graph isomorphism and any other operation that does not
change the Nim value} is an upper bound on the Nim value of a position.  As
a result, edge-transitive graphs have a maximum Nim value of 2: a player
could delete any edge (it does not matter which one), or contract any edge,
and in the maximizing case, one of those options gives Nim value 0, one
gives Nim value 1, and the edge-transitive starting graph can have Nim value
2.  More generally, if there are $k$ orbits of edges under the automorphism
group of $G$, then $\mathcal{G}(G)\le 2k$.

But there can be other equivalent moves not captured by the graph
automorphism group.  For instance, deleting any edge in a chain of degree-2
vertices will yield \emph{equivalent} but not necessarily \emph{isomorphic}
graphs regardless of which edge is deleted, because the remaining edges in
the chain all become single-edge blocks, and then only the parity of how
many of them there are is relevant to the Nim value.  Recognizing moves that
are equivalent in this way can tighten the bound a little.

On the other side, the computer results of the next section include
biconnected graphs with Nim values as large as 25.  By the definition of Nim
values, existence of any value implies existence of all smaller values. 
Combining power-of-two values from 1 to 16 with the Nim sum operation allows
the construction of non-biconnected graphs with arbitrary Nim values from 0
to 31.  It seems intuitively reasonable that graphs ought to exist with
arbitrarily large Nim values, but no Nim value greater than 31 has actually
been proven to occur.


\section{Computer experiments}

We implemented the obvious dynamic programming algorithm for computing Nim
values of graphs: namely recursively computing the Nim values of all
options and taking the mex of them, while memoizing computed results in a
hash table indexed by a canonically-labelled representation of the graph. 
Our software has a client-server architecture intended for use on a
multicore machine.  Each client reads graphs from its input and computes
their Nim values as follows:
\begin{itemize}
  \item Detect a few small basis cases (such as graphs with at most
    three edges) and return hardcoded answers for them.
  \item If the graph is not biconnected:  split it into blocks,
    solve those separately, and compute the Nim sum.
  \item When working on a biconnected graph, canonically label it.
  \item Check a local per-client cache (hash table of $2^{25}$ entries,
    roughly 1G of RAM).
  \item If the answer is not in the local hash table: query the
    database server.
  \item If not on the database server:  recursively compute all the Nim
    values of options, and take their mex.
  \item If we did a recursive examination of options:  store the result on
    the database server and in the local hash table, overwriting any
    colliding item in the local hash slot.
\end{itemize}

We used the Tokyo Tyrant key-value store~\cite{FAL:Tokyo} as the central
database server, and wrote client programs in C with
nauty~\cite{McKay:nauty} for canonical labelling.  Although the recursion
rule is different, this general approach of memoized recursion over smaller
graphs is essentially the same as that used by our cycle-counting software
(``ECCHI,'' the Enhanced Cycle Counter and Hamiltonian Integrator) in a
previous project~\cite{Durocher:Cycle}, and we were able to re-use some of
that code.  We used the graph utilities included with nauty to generate sets
of graphs to feed into the computation.

Bearing in mind the difficulty of verifying correctness of final answers for
larger graphs, we spent significant effort on testing the code.  The final
test suite achieves 100\%\ source line coverage of our client software
(excluding third-party material and assertion-failed branches) and covers a
wide range of cases reasonably expected to be relevant to correctness.  For
instance, one test computes the Nim values of all 8-vertex biconnected
graphs (without connecting to the database server), then does it again with
the graphs in a pseudorandomly permuted order, and checks that the results
are the same for all of the graphs.  Since the computation for each graph
depends on the intermediate values stored in the local cache by previous
computations, this test implies finding the answer for each graph by two
different computation trees.  We also ran our tests inside
Valgrind~\cite{Nethercote:Valgrind} to guard
against uninitialized values and other kinds of undefined behaviour.  The
results from our software agree with all our hand calculations (including on
all graphs of up to five vertices) and theoretical results (including
some that were not known when the software was written).

We ran our experiments on one node of a Linux cluster at the IT University
of Copenhagen, with four real Intel CPU cores (eight virtual by
``hyper-threading'') running at 3.60GHz, and 32G of RAM.  We started with
the database on a 250G solid-state drive, switching to a magnetic hard drive
in the final stages when space for the database (including temporary working
space needed by Tokyo Tyrant's ``optimization'' process) ran out on the SSD.

We computed Nim values for the following graphs:
\begin{itemize}
\item Biconnected graphs with 3 to 11 vertices (910914360 graphs
  total).
\item Planar biconnected graphs with 3 to 12 vertices (169178844 graphs
  total).
\item Triangle-free biconnected graphs with 4 to 13 vertices (10757199
  graphs total).
\item Graphs of girth at least five, and biconnected, with 5 to 15 vertices
  (342385 graphs).
\item Cubic triangle-free biconnected graphs with up to 16 vertices (928
  graphs).
\item Complete bipartite graphs $K_{p,q}$ with $p$ and $q$ at most 20 and at
  most 48 edges.
\end{itemize}

All but the largest vertex counts of these experiments ran within about four
days.  There is no single precise number because we repeated the experiments
several times under varying conditions, both to confirm the results and to
test different software configurations.  The largest sizes, which involved
more graphs and slower access to larger files, consumed more like two or
three weeks of computation.

Table~\ref{tab:nim-per-vertex} shows the maximum Nim value known for a
biconnected graph, and the Nim value of the complete graph, for each
value of $n$, the number of vertices.  The case $n=4$ is the only one for
which a non-biconnected graph is known to achieve a greater Nim value (3,
for a triangle plus an edge) than any biconnected graph.  Complete graphs
are interesting for their lack of pattern.  We know the values are
necessarily in $\{0,1,2\}$ because complete graphs are edge-transitive, and
$\mathcal{G}(K_n)\ne\mathcal{G}(K_{n-1})$ because the next smaller complete
graph is always an option; but there is no obvious way to calculate
$\mathcal{G}(K_n)$ faster than recursing over \emph{all} smaller graphs. 

\begin{table}
\begin{tabular}{r|r}
$n$ & $\max \mathcal{G}(G)$ \\ \hline
1 & 0 \\
2 & 1 \\
3 & 2 \\
4 & 1 \\
5 & 4 \\
6 & 6 \\
7 & 8 \\
8 & 13 \\
9 & 18 \\
10 & 22 \\
11 & 25
\end{tabular}
\qquad
\begin{tabular}{r|r}
$n$ & $\mathcal{G}(K_n)$ \\ \hline
1 & 0 \\
2 & 1 \\
3 & 2 \\
4 & 0 \\
5 & 1 \\
6 & 2 \\
7 & 0 \\
8 & 2 \\
9 & 0 \\
10 & 1 \\
11 & 2
\end{tabular}
\caption{Maximum Nim values of biconnected graphs, and Nim values of
complete graphs, by number of vertices}\label{tab:nim-per-vertex}
\end{table}

Searches of the sequences from Table~\ref{tab:nim-per-vertex} and near
variations in the On-Line Encyclopedia of Integer Sequences~\cite{OEIS} turn
up very little.  Some appealing hits are excluded by theoretical
considerations; for instance, the fact that $\max \mathcal{G}(G)$ for any
number of vertices $n$ cannot exceed $\binom{n}{2}$, the maximum number of
edges.  The most exciting search result is that the indices of zeroes in
$\mathcal{G}(K_n)$, namely $1,4,7,9,\ldots$, agree with sequence A007066 for
all known values.  That sequence is described as ``$a(n)=1+\lceil
(n-1)\phi^2\rceil$, $\phi=(1+\sqrt{5})/2$.'' The next few terms are $12, 15,
17, 20, 22, 25,\ldots$ The citations for A007066 include Morrison's
work~\cite{Morrison:Stolarsky} on Wythoff pairs, which come from the
analysis of Wythoff's well-known game~\cite{Wythoff:Modification}.  But
exactly how the Golden Ratio and Wythoff's game would be linked to Graph
Nimors is not clear, and there are so few terms of the sequence known as to
make any connection unreliable.  It would be very interesting, and may
possibly be computationally feasible, to determine $\mathcal{G}(K_{12})$. 
If the link to A007066 is genuine, that ought to be 0.

We collected the complete distribution of Nim values for each combination of
vertex count ($n$) and edge count ($m$); this data is presented in
Appendix~\ref{app:stats}.  In general, the pattern was that for any
combination of $n$ and $m$, there would be just a few very common Nim values
accounting for nearly all the biconnected graphs with those parameters.  The
distribution for $n=10$, $m=23$ shown in Figure~\ref{fig:1023-dist} is a
typical example, with Nim values 1 and 5 accounting for approximately 85\%\
of the graphs.

\begin{figure}
{\small\begin{tabular}{r|r}
$\mathcal{G}$ & \#\ graphs \\ \hline
0 & 23059 \\
1 & 724676 \\
2 & 8889 \\
3 & 418 \\
4 & 7312 \\
5 & 312881 \\
6 & 8679 \\
7 & 23683 \\
8 & 30896 \\
9 & 31990 \\
10 & 21243 \\
11 & 14501 \\
12 & 9004 \\
13 & 4810 \\
14 & 2071 \\
15 & 301 \\
16 & 17
\end{tabular}}
\quad
\raisebox{-1.4in}{\includegraphics[scale=0.87]{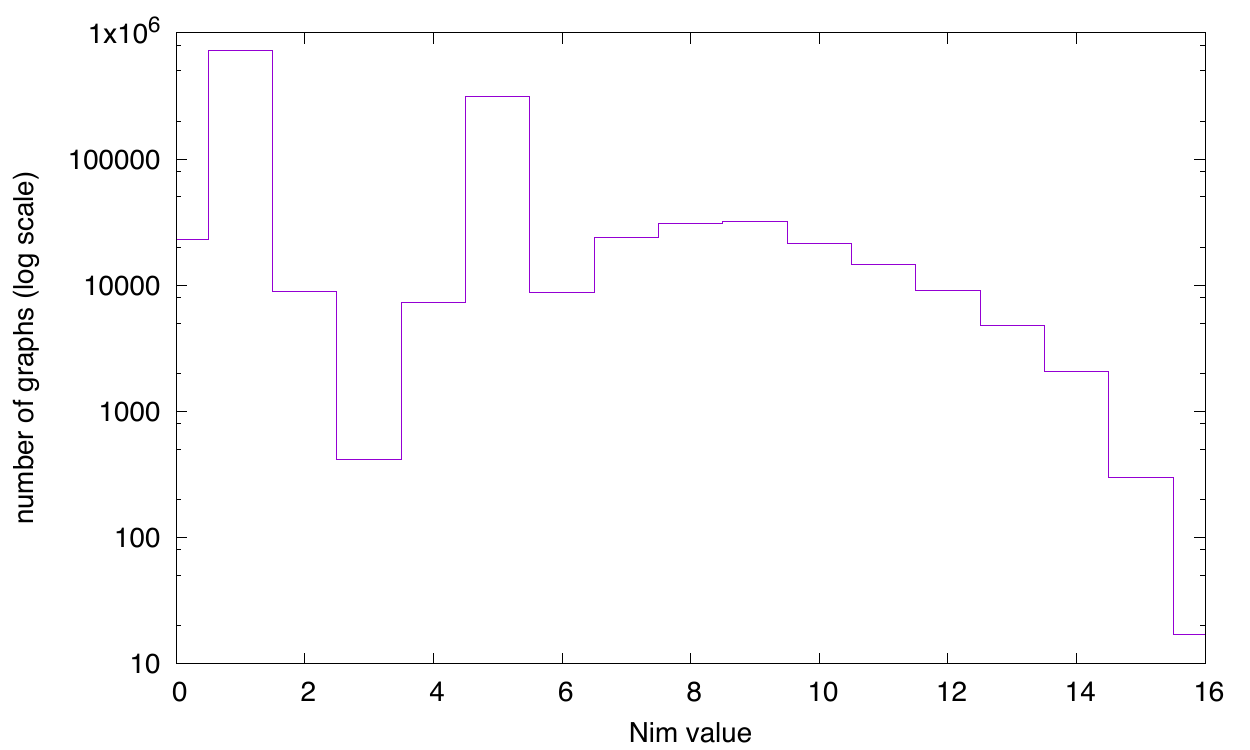}}
\caption{Distribution of Nim values for the 1224430 biconnected graphs of 10
vertices and 23 edges.}\label{fig:1023-dist}
\end{figure}

Values common for a given $m$ are usually very rare for the next larger $m$.
All deletion moves leave the graph with one less edge, and
in a dense graph deletion moves usually leave the graph biconnected and with
the same number of (non-isolated) vertices.  Similarly, all contraction
moves reduce the vertex count by one and the edge count by at least one; in
a sparse graph, a contraction move will usually remove exactly one edge. 
Thus, if a given Nim value is common for (vertex, edge) counts $(n,m-1)$ or
$(n-1,m-1)$, and to a lesser extent, $n-1$ and even smaller $m$, then we
should expect that value to be uncommon for $(n,m)$.  The options for a
graph of a given size will usually include a representative sample of the
graphs one edge smaller.  This interaction between parameter values may
lead to
some of the same kinds of periodic behaviour seen in simpler take-and-break
games on piles of stones~\cite[Chapter 4]{Berlekamp:Winning}, even if only
as a matter of usual-case statistics not guaranteed for all graphs of a
given $n$ and $m$.

Table~\ref{tab:complete-bipartite} shows all the experimentally-calculated
values of $\mathcal{G}(K_{p,q})$; that is, Nim values of complete bipartite
graphs.  The evident patterns in the first two rows are proven
($K_{1,q}$ because the graphs are acyclic, $K_{2,q}$ by
Theorem~\ref{thm:property-s}), but the others remain theoretically open. 
This is another class in which all the graphs are edge-transitive, so the
values are constrained to $\{0,1,2\}$.

\begin{table}
{\setlength{\tabcolsep}{0.5em}
\begin{tabular}{cc|cccccccccccccccccccc}
& & $q$ \\
& & 1 & 2 & 3 & 4 & 5 & 6 & 7 & 8 & 9 & 10 & 11 & 12 & 13 & 14 & 15 & 16 & 16 & 18 & 19 & 20 \\ \hline
$p$ & 1 & 1 & 0 & 1 & 0 & 1 & 0 & 1 & 0 & 1 & 0 & 1 & 0 & 1 & 0 & 1 & 0 & 1 & 0 & 1 & 0 \\
& 2 & & 0 & 0 & 0 & 0 & 0 & 0 & 0 & 0 & 0 & 0 & 0 & 0 & 0 & 0 & 0 & 0 & 0 & 0 & 0 \\
& 3 & & & 1 & 0 & 1 & 0 & 1 & 0 & 1 & 0 & 1 & 0 & 1 & 0 & 1 & 0 \\
& 4 & & & & 2 & 0 & 2 & 0 & 2 & 0 & 2 & 0 & 2 \\
& 5 & & & & & 0 & 0 & 0 & 0 & 0 \\
& 6 & & & & & & 1 & 0 & 1
\end{tabular}}
\caption{Nim values of complete bipartite graphs.}
\label{tab:complete-bipartite}
\end{table}


\section{The Parity Heuristic}\label{sec:ph}

A deletion move always subtracts one from the total number of edges in the
graph.  In a sparse graph, a randomly chosen contraction move will probably
be on an edge not in any triangle, so it will also subtract exactly one from
the total number of edges.  In a dense graph, a contraction move may remove
many edges, but it still seems that a random contraction would as likely as
not be on an edge that is part of an even number of triangles, so that it
will reduce the edge count by an odd number.  Thus, if we knew nothing about
strategy, we might expect that at least within some kind of approximation,
players would remove an odd number of edges on every move and we could
evaluate whether a position favours the next or previous player simply by
looking at the parity of the number of edges remaining.  The following
definition is a stronger form of that intuitive expectation.

\begin{definition}
The \emph{Parity Heuristic} (PH) is the proposition that for a graph $G$
with $m$ edges, $\mathcal{G}(G)$ is 0 if $m$ is even and 1 if $m$ is odd.
\end{definition}

Since graphs of Nim value other than $0$ and $1$ exist, PH fails as a
complete analysis of the game.  However, the computer results, and experience
with human play, suggest that PH holds for very many graphs.

The Parity Heuristic is proven to hold in these cases:
\begin{itemize}
\item acyclic graphs (all moves leave the graph acyclic and with one less
edge, induction down to edgeless graphs);
\item cycles except $C_3$ (as described in Subsection~\ref{sub:easy-cases});
\item fused cycle pairs, if neither is a triangle
(Theorem~\ref{thm:fused-cycle});
\item $K_{2,q}$ for any $q$ (these graphs have property S and even edge
counts, see Theorem~\ref{thm:property-s}); and
\item graphs of more than one block, if it holds for each of the blocks
(by the Sprague-Grundy Theorem).
\end{itemize}

For graphs with property S, we
have Theorem~\ref{thm:property-s} that $\mathcal{G}(G)=0$ if and only if $m$
is even.  That is equivalent to PH when the number of edges is even, but a
little weaker when it is odd.

We conjecture that PH holds for:
\begin{itemize}
\item graphs with property S and an odd number of edges (not all existing
computer results have been searched for this, but it is a reasonable
extension of the theoretical results);
\item $K_{3,q}$ for any $q$ (no counterexamples up to $K_{3,16}$);
\item graphs of girth at least 5 (no counterexamples up to $n=15$); and
\item cubic triangle-free graphs (no counterexamples up to $n=16$).
\end{itemize}

It is known not to hold in general for:
\begin{itemize}
\item all graphs (smallest counterexample $C_3$, Nim value 2);
\item cubic graphs (smallest counterexample the triangular
prism graph, with $n=6$, $m=9$, Nim value 0);
\item triangle-free graphs (smallest counterexample $K_{4,4}$,
Nim value 2); nor
\item complete graphs ($C_3$ is a counterexample, but there are several
others known also).
\end{itemize}

The Parity Heuristic is not proven to always fail for any interesting
infinite classes of graphs.  However, for all known cases of $K_{p,q}$ with
$p$ and $q$ both at least 4, the Nim value is nonzero if and only if $p$ and
$q$ are both even, which contradicts PH whenever $p+q$ is even.


\section{Further thoughts}

We have described the game of Graph Nimors and some theoretical and
experimental results on strategy for it.  Many natural questions
remain open.

All the conjectures regarding the Parity Heuristic in
Section~\ref{sec:ph} seem good targets for theoretical work.  We are
especially interested in the girth-5 case, which seems like it should be
easy to prove.  Proving Nim values for well-behaved infinite classes of
graphs, such as $K_{p,q}$ with fixed constant $p$ such as 3 or 4, also seems
like a bite-sized problem.  Any result on $\mathcal{G}(K_n)$ (that
is, the Nim value of the arbitrary-sized complete graph) would be
interesting, but may be difficult; in particular, the coincidence with OEIS
sequence A007066~\cite{OEIS}, which is related to the Golden Ratio and
Wythoff's Nim-like game, would be interesting to confirm or disprove.  Just
computing $\mathcal{G}(K_{12})$, currently known to be either 0 or
1, could either lend additional support to that connection (if the answer is
0) or immediately disprove it (if the answer is 1); and that seems to be a
large computational task, but within the range of possibility, given some
improvements to software and hardware.

The experimental side of this work revealed some deficiencies in Tokyo
Tyrant's ability to handle databases on magnetic disk as opposed to SSD, and
other high-performance key-value stores suitable for external-memory
databases are surprisingly few.  Popular ``noSQL databases'' are frequently
designed for smaller numbers of larger records, or to operate only in main
memory.  Building a key-value store capable of handling a random access
pattern on magnetic disk with many billions of very small records
(presumably, batching requests from many parallel threads to make the best
of each disk seek operation) is an interesting software engineering problem.

It is reasonable to guess that calculating the Nim value of a graph with
respect to Graph Nimors should be PSPACE-complete, but that remains
unproven.  Constraining the moves, for instance by fixing a sequence of the
edges and requiring players to follow that sequence, might create a variant
for which hardness is easier to prove.  Much of the theoretical difficulty
comes from the fact that there is currently no known way to split a graph
into smaller parts with predictable relations between the Nim values of the
parts, except to split it into blocks, at which point the blocks' values
affect each other only through the Nim sum operation.  Having any other way
to localize the effects of changes in the graph would help support
construction of gadgets for a hardness proof.  Constructions for arbitrarily
large biconnected graphs with specified Nim values; arbitrarily large Nim
values; or a proof that arbitrarily large Nim values are not possible; might
contribute usefully to the hardness question as well as being interesting in
themselves.

Many variations of Graph Nimors are possible.  The \emph{mis\`{e}re}
variation is obvious, and could be expected to yield as much complicated
theory as any other impartial \emph{mis\`{e}re} game.  One could make Graph
Nimors partisan by requiring one player to always delete and one to always
contract.  In a graph of large girth with few cycles, the deleting player
may be able to break all the cycles before the contracting player can form
any triangles, thus forcing the game to be determined by parity of number of
edges; but if that is not in the deleting player's interest, or if the girth
is small or number of cycles large, the result is not clear.  When we first
invented this game, we were concerned that it might turn out to be too easy
under the basic rules presented here, and considered adding constraints like
``no move is allowed that would leave the graph planar.'' Although
apparently unnecessary to create a difficult game, such a constraint might
be interesting as a way to link nimors and minors.


\bibliography{nimors}


\appendix

\section{Distributions of Nim values for biconnected graphs}\label{app:stats}

This appendix gives the counts of Nim values observed for all biconnected
graphs with between 3 and 11 vertices, sorted with the most common values at
the top.

\subsection{3 vertices}

\begin{tabular}{r|l}
\multicolumn{2}{c}{$n=3$\quad$m=3$}\\
 $\mathcal{G}$ & count \\ \hline
  2 & 1 
\end{tabular}

\subsection{4 vertices}

\begin{tabular}{r|l}
\multicolumn{2}{c}{$n=4$\quad$m=4$}\\
 $\mathcal{G}$ & count \\ \hline
  0 & 1 
\end{tabular}
\begin{tabular}{r|l}
\multicolumn{2}{c}{$n=4$\quad$m=5$}\\
 $\mathcal{G}$ & count \\ \hline
  1 & 1 
\end{tabular}
\begin{tabular}{r|l}
\multicolumn{2}{c}{$n=4$\quad$m=6$}\\
 $\mathcal{G}$ & count \\ \hline
  0 & 1 
\end{tabular}

\subsection{5 vertices}

\begin{tabular}{r|l}
\multicolumn{2}{c}{$n=5$\quad$m=5$}\\
 $\mathcal{G}$ & count \\ \hline
  1 & 1 
\end{tabular}
\begin{tabular}{r|l}
\multicolumn{2}{c}{$n=5$\quad$m=6$}\\
 $\mathcal{G}$ & count \\ \hline
  4 & 1 \\
  0 & 1 
\end{tabular}
\begin{tabular}{r|l}
\multicolumn{2}{c}{$n=5$\quad$m=7$}\\
 $\mathcal{G}$ & count \\ \hline
  2 & 3 
\end{tabular}
\begin{tabular}{r|l}
\multicolumn{2}{c}{$n=5$\quad$m=8$}\\
 $\mathcal{G}$ & count \\ \hline
  4 & 1 \\
  3 & 1 
\end{tabular}
\begin{tabular}{r|l}
\multicolumn{2}{c}{$n=5$\quad$m=9$}\\
 $\mathcal{G}$ & count \\ \hline
  2 & 1 
\end{tabular}
\begin{tabular}{r|l}
\multicolumn{2}{c}{$n=5$\quad$m=10$}\\
 $\mathcal{G}$ & count \\ \hline
  1 & 1 
\end{tabular}

\subsection{6 vertices}

\begin{tabular}{r|l}
\multicolumn{2}{c}{$n=6$\quad$m=6$}\\
 $\mathcal{G}$ & count \\ \hline
  0 & 1 
\end{tabular}
\begin{tabular}{r|l}
\multicolumn{2}{c}{$n=6$\quad$m=7$}\\
 $\mathcal{G}$ & count \\ \hline
  1 & 2 \\
  2 & 1 
\end{tabular}
\begin{tabular}{r|l}
\multicolumn{2}{c}{$n=6$\quad$m=8$}\\
 $\mathcal{G}$ & count \\ \hline
  3 & 4 \\
  0 & 4 \\
  1 & 1 
\end{tabular}
\begin{tabular}{r|l}
\multicolumn{2}{c}{$n=6$\quad$m=9$}\\
 $\mathcal{G}$ & count \\ \hline
  1 & 10 \\
  5 & 2 \\
  4 & 1 \\
  0 & 1 
\end{tabular}
\begin{tabular}{r|l}
\multicolumn{2}{c}{$n=6$\quad$m=10$}\\
 $\mathcal{G}$ & count \\ \hline
  0 & 6 \\
  3 & 2 \\
  6 & 1 \\
  5 & 1 \\
  4 & 1 \\
  1 & 1 
\end{tabular}
\begin{tabular}{r|l}
\multicolumn{2}{c}{$n=6$\quad$m=11$}\\
 $\mathcal{G}$ & count \\ \hline
  1 & 5 \\
  5 & 2 \\
  3 & 1 
\end{tabular}
\begin{tabular}{r|l}
\multicolumn{2}{c}{$n=6$\quad$m=12$}\\
 $\mathcal{G}$ & count \\ \hline
  0 & 4 \\
  2 & 1 
\end{tabular}
\begin{tabular}{r|l}
\multicolumn{2}{c}{$n=6$\quad$m=13$}\\
 $\mathcal{G}$ & count \\ \hline
  4 & 1 \\
  3 & 1 
\end{tabular}
\begin{tabular}{r|l}
\multicolumn{2}{c}{$n=6$\quad$m=14$}\\
 $\mathcal{G}$ & count \\ \hline
  0 & 1 
\end{tabular}
\begin{tabular}{r|l}
\multicolumn{2}{c}{$n=6$\quad$m=15$}\\
 $\mathcal{G}$ & count \\ \hline
  2 & 1 
\end{tabular}

\subsection{7 vertices}

\begin{tabular}{r|l}
\multicolumn{2}{c}{$n=7$\quad$m=7$}\\
 $\mathcal{G}$ & count \\ \hline
  1 & 1 
\end{tabular}
\begin{tabular}{r|l}
\multicolumn{2}{c}{$n=7$\quad$m=8$}\\
 $\mathcal{G}$ & count \\ \hline
  0 & 3 \\
  3 & 1 
\end{tabular}
\begin{tabular}{r|l}
\multicolumn{2}{c}{$n=7$\quad$m=9$}\\
 $\mathcal{G}$ & count \\ \hline
  1 & 10 \\
  2 & 7 \\
  5 & 2 \\
  4 & 1 
\end{tabular}
\begin{tabular}{r|l}
\multicolumn{2}{c}{$n=7$\quad$m=10$}\\
 $\mathcal{G}$ & count \\ \hline
  0 & 25 \\
  4 & 16 \\
  6 & 3 \\
  7 & 2 \\
  1 & 2 \\
  5 & 1 \\
  3 & 1 
\end{tabular}
\begin{tabular}{r|l}
\multicolumn{2}{c}{$n=7$\quad$m=11$}\\
 $\mathcal{G}$ & count \\ \hline
  2 & 66 \\
  1 & 6 \\
  8 & 3 \\
  5 & 3 \\
  7 & 2 \\
  3 & 2 
\end{tabular}
\begin{tabular}{r|l}
\multicolumn{2}{c}{$n=7$\quad$m=12$}\\
 $\mathcal{G}$ & count \\ \hline
  4 & 38 \\
  0 & 20 \\
  3 & 16 \\
  7 & 6 \\
  8 & 5 \\
  6 & 5 \\
  1 & 3 \\
  5 & 1 
\end{tabular}
\begin{tabular}{r|l}
\multicolumn{2}{c}{$n=7$\quad$m=13$}\\
 $\mathcal{G}$ & count \\ \hline
  2 & 57 \\
  1 & 6 \\
  7 & 5 \\
  5 & 5 \\
  9 & 4 \\
  8 & 2 \\
  6 & 1 \\
  4 & 1 
\end{tabular}
\begin{tabular}{r|l}
\multicolumn{2}{c}{$n=7$\quad$m=14$}\\
 $\mathcal{G}$ & count \\ \hline
  4 & 27 \\
  3 & 15 \\
  0 & 6 \\
  6 & 5 \\
  1 & 4 \\
  5 & 2 
\end{tabular}
\begin{tabular}{r|l}
\multicolumn{2}{c}{$n=7$\quad$m=15$}\\
 $\mathcal{G}$ & count \\ \hline
  2 & 18 \\
  1 & 8 \\
  6 & 3 \\
  0 & 3 \\
  7 & 2 \\
  5 & 2 \\
  4 & 1 \\
  3 & 1 
\end{tabular}
\begin{tabular}{r|l}
\multicolumn{2}{c}{$n=7$\quad$m=16$}\\
 $\mathcal{G}$ & count \\ \hline
  4 & 5 \\
  5 & 4 \\
  6 & 3 \\
  3 & 3 \\
  1 & 2 \\
  0 & 2 \\
  8 & 1 
\end{tabular}
\begin{tabular}{r|l}
\multicolumn{2}{c}{$n=7$\quad$m=17$}\\
 $\mathcal{G}$ & count \\ \hline
  1 & 5 \\
  2 & 4 \\
  0 & 1 
\end{tabular}
\begin{tabular}{r|l}
\multicolumn{2}{c}{$n=7$\quad$m=18$}\\
 $\mathcal{G}$ & count \\ \hline
  5 & 1 \\
  4 & 1 \\
  3 & 1 \\
  1 & 1 \\
  0 & 1 
\end{tabular}
\begin{tabular}{r|l}
\multicolumn{2}{c}{$n=7$\quad$m=19$}\\
 $\mathcal{G}$ & count \\ \hline
  6 & 1 \\
  1 & 1 
\end{tabular}
\begin{tabular}{r|l}
\multicolumn{2}{c}{$n=7$\quad$m=20$}\\
 $\mathcal{G}$ & count \\ \hline
  3 & 1 
\end{tabular}
\begin{tabular}{r|l}
\multicolumn{2}{c}{$n=7$\quad$m=21$}\\
 $\mathcal{G}$ & count \\ \hline
  0 & 1 
\end{tabular}

\subsection{8 vertices}

\begin{tabular}{r|l}
\multicolumn{2}{c}{$n=8$\quad$m=8$}\\
 $\mathcal{G}$ & count \\ \hline
  0 & 1 
\end{tabular}
\begin{tabular}{r|l}
\multicolumn{2}{c}{$n=8$\quad$m=9$}\\
 $\mathcal{G}$ & count \\ \hline
  1 & 5 \\
  2 & 1 
\end{tabular}
\begin{tabular}{r|l}
\multicolumn{2}{c}{$n=8$\quad$m=10$}\\
 $\mathcal{G}$ & count \\ \hline
  0 & 25 \\
  3 & 9 \\
  4 & 6 
\end{tabular}
\begin{tabular}{r|l}
\multicolumn{2}{c}{$n=8$\quad$m=11$}\\
 $\mathcal{G}$ & count \\ \hline
  1 & 69 \\
  2 & 34 \\
  5 & 26 \\
  6 & 16 \\
  8 & 6 \\
  4 & 5 \\
  0 & 4 \\
  3 & 1 
\end{tabular}
\begin{tabular}{r|l}
\multicolumn{2}{c}{$n=8$\quad$m=12$}\\
 $\mathcal{G}$ & count \\ \hline
  3 & 301 \\
  0 & 75 \\
  4 & 15 \\
  6 & 12 \\
  7 & 6 \\
  5 & 6 \\
  1 & 6 \\
  9 & 3 \\
  8 & 3 \\
  2 & 2 
\end{tabular}
\begin{tabular}{r|l}
\multicolumn{2}{c}{$n=8$\quad$m=13$}\\
 $\mathcal{G}$ & count \\ \hline
  1 & 468 \\
  5 & 139 \\
  0 & 59 \\
  6 & 31 \\
  4 & 28 \\
  2 & 26 \\
  7 & 17 \\
  8 & 8 \\
  10 & 3 \\
  9 & 1 
\end{tabular}
\begin{tabular}{r|l}
\multicolumn{2}{c}{$n=8$\quad$m=14$}\\
 $\mathcal{G}$ & count \\ \hline
  3 & 576 \\
  0 & 209 \\
  6 & 105 \\
  8 & 43 \\
  7 & 40 \\
  4 & 29 \\
  9 & 23 \\
  5 & 23 \\
  2 & 9 \\
  10 & 9 \\
  1 & 9 \\
  11 & 1 
\end{tabular}
\begin{tabular}{r|l}
\multicolumn{2}{c}{$n=8$\quad$m=15$}\\
 $\mathcal{G}$ & count \\ \hline
  1 & 694 \\
  5 & 240 \\
  0 & 99 \\
  6 & 31 \\
  4 & 29 \\
  2 & 29 \\
  7 & 27 \\
  9 & 20 \\
  8 & 13 \\
  10 & 8 \\
  11 & 6 \\
  3 & 1 
\end{tabular}
\begin{tabular}{r|l}
\multicolumn{2}{c}{$n=8$\quad$m=16$}\\
 $\mathcal{G}$ & count \\ \hline
  3 & 383 \\
  0 & 270 \\
  6 & 174 \\
  8 & 88 \\
  7 & 69 \\
  10 & 28 \\
  2 & 27 \\
  5 & 25 \\
  9 & 23 \\
  4 & 16 \\
  1 & 6 \\
  11 & 4 \\
  12 & 1 
\end{tabular}
\begin{tabular}{r|l}
\multicolumn{2}{c}{$n=8$\quad$m=17$}\\
 $\mathcal{G}$ & count \\ \hline
  1 & 390 \\
  5 & 231 \\
  9 & 49 \\
  0 & 39 \\
  7 & 33 \\
  2 & 31 \\
  4 & 27 \\
  8 & 26 \\
  6 & 20 \\
  10 & 18 \\
  11 & 14 \\
  3 & 7 
\end{tabular}
\begin{tabular}{r|l}
\multicolumn{2}{c}{$n=8$\quad$m=18$}\\
 $\mathcal{G}$ & count \\ \hline
  0 & 211 \\
  3 & 165 \\
  8 & 60 \\
  6 & 57 \\
  7 & 45 \\
  2 & 27 \\
  10 & 21 \\
  4 & 11 \\
  9 & 9 \\
  11 & 6 \\
  5 & 5 \\
  1 & 3 \\
  12 & 2 
\end{tabular}
\begin{tabular}{r|l}
\multicolumn{2}{c}{$n=8$\quad$m=19$}\\
 $\mathcal{G}$ & count \\ \hline
  5 & 92 \\
  1 & 81 \\
  9 & 54 \\
  2 & 34 \\
  7 & 27 \\
  4 & 22 \\
  6 & 17 \\
  3 & 17 \\
  11 & 13 \\
  8 & 12 \\
  10 & 8 \\
  12 & 4 \\
  0 & 3 \\
  13 & 2 
\end{tabular}
\begin{tabular}{r|l}
\multicolumn{2}{c}{$n=8$\quad$m=20$}\\
 $\mathcal{G}$ & count \\ \hline
  0 & 103 \\
  3 & 44 \\
  7 & 19 \\
  8 & 13 \\
  6 & 13 \\
  2 & 10 \\
  10 & 4 \\
  4 & 3 \\
  1 & 3 \\
  9 & 1 \\
  12 & 1 \\
  11 & 1 
\end{tabular}
\begin{tabular}{r|l}
\multicolumn{2}{c}{$n=8$\quad$m=21$}\\
 $\mathcal{G}$ & count \\ \hline
  5 & 24 \\
  2 & 17 \\
  4 & 16 \\
  9 & 11 \\
  6 & 11 \\
  1 & 10 \\
  7 & 8 \\
  8 & 7 \\
  3 & 6 \\
  0 & 2 
\end{tabular}
\begin{tabular}{r|l}
\multicolumn{2}{c}{$n=8$\quad$m=22$}\\
 $\mathcal{G}$ & count \\ \hline
  0 & 25 \\
  3 & 12 \\
  2 & 6 \\
  8 & 3 \\
  7 & 2 \\
  6 & 2 \\
  1 & 2 \\
  5 & 1 \\
  4 & 1 \\
  10 & 1 
\end{tabular}
\begin{tabular}{r|l}
\multicolumn{2}{c}{$n=8$\quad$m=23$}\\
 $\mathcal{G}$ & count \\ \hline
  2 & 8 \\
  4 & 6 \\
  5 & 4 \\
  7 & 3 \\
  1 & 2 \\
  8 & 1 
\end{tabular}
\begin{tabular}{r|l}
\multicolumn{2}{c}{$n=8$\quad$m=24$}\\
 $\mathcal{G}$ & count \\ \hline
  0 & 6 \\
  7 & 2 \\
  9 & 1 \\
  3 & 1 \\
  1 & 1 
\end{tabular}
\begin{tabular}{r|l}
\multicolumn{2}{c}{$n=8$\quad$m=25$}\\
 $\mathcal{G}$ & count \\ \hline
  2 & 4 \\
  1 & 1 
\end{tabular}
\begin{tabular}{r|l}
\multicolumn{2}{c}{$n=8$\quad$m=26$}\\
 $\mathcal{G}$ & count \\ \hline
  4 & 2 
\end{tabular}
\begin{tabular}{r|l}
\multicolumn{2}{c}{$n=8$\quad$m=27$}\\
 $\mathcal{G}$ & count \\ \hline
  1 & 1 
\end{tabular}
\begin{tabular}{r|l}
\multicolumn{2}{c}{$n=8$\quad$m=28$}\\
 $\mathcal{G}$ & count \\ \hline
  2 & 1 
\end{tabular}

\subsection{9 vertices}

\begin{tabular}{r|l}
\multicolumn{2}{c}{$n=9$\quad$m=9$}\\
 $\mathcal{G}$ & count \\ \hline
  1 & 1 
\end{tabular}
\begin{tabular}{r|l}
\multicolumn{2}{c}{$n=9$\quad$m=10$}\\
 $\mathcal{G}$ & count \\ \hline
  0 & 6 \\
  3 & 1 
\end{tabular}
\begin{tabular}{r|l}
\multicolumn{2}{c}{$n=9$\quad$m=11$}\\
 $\mathcal{G}$ & count \\ \hline
  1 & 44 \\
  2 & 20 \\
  5 & 6 
\end{tabular}
\begin{tabular}{r|l}
\multicolumn{2}{c}{$n=9$\quad$m=12$}\\
 $\mathcal{G}$ & count \\ \hline
  0 & 178 \\
  3 & 161 \\
  4 & 59 \\
  7 & 12 \\
  2 & 8 \\
  6 & 5 \\
  1 & 5 \\
  5 & 4 \\
  8 & 1 
\end{tabular}
\begin{tabular}{r|l}
\multicolumn{2}{c}{$n=9$\quad$m=13$}\\
 $\mathcal{G}$ & count \\ \hline
  2 & 859 \\
  1 & 482 \\
  5 & 162 \\
  6 & 68 \\
  4 & 44 \\
  8 & 41 \\
  7 & 28 \\
  0 & 21 \\
  10 & 11 \\
  9 & 10 \\
  3 & 3 
\end{tabular}
\begin{tabular}{r|l}
\multicolumn{2}{c}{$n=9$\quad$m=14$}\\
 $\mathcal{G}$ & count \\ \hline
  0 & 2318 \\
  4 & 1428 \\
  7 & 237 \\
  3 & 224 \\
  6 & 206 \\
  5 & 159 \\
  8 & 73 \\
  1 & 56 \\
  9 & 50 \\
  2 & 24 \\
  10 & 19 \\
  11 & 2 
\end{tabular}
\begin{tabular}{r|l}
\multicolumn{2}{c}{$n=9$\quad$m=15$}\\
 $\mathcal{G}$ & count \\ \hline
  2 & 7487 \\
  1 & 622 \\
  5 & 520 \\
  6 & 375 \\
  7 & 324 \\
  8 & 297 \\
  9 & 121 \\
  4 & 68 \\
  3 & 57 \\
  10 & 57 \\
  0 & 40 \\
  11 & 13 
\end{tabular}
\begin{tabular}{r|l}
\multicolumn{2}{c}{$n=9$\quad$m=16$}\\
 $\mathcal{G}$ & count \\ \hline
  4 & 7324 \\
  0 & 5733 \\
  3 & 733 \\
  7 & 607 \\
  6 & 561 \\
  5 & 426 \\
  8 & 398 \\
  9 & 270 \\
  1 & 258 \\
  10 & 136 \\
  11 & 71 \\
  2 & 17 \\
  12 & 8 
\end{tabular}
\begin{tabular}{r|l}
\multicolumn{2}{c}{$n=9$\quad$m=17$}\\
 $\mathcal{G}$ & count \\ \hline
  2 & 18029 \\
  7 & 853 \\
  6 & 833 \\
  5 & 804 \\
  1 & 773 \\
  8 & 572 \\
  9 & 362 \\
  10 & 212 \\
  3 & 158 \\
  11 & 92 \\
  0 & 71 \\
  4 & 59 \\
  12 & 24 \\
  13 & 2 
\end{tabular}
\begin{tabular}{r|l}
\multicolumn{2}{c}{$n=9$\quad$m=18$}\\
 $\mathcal{G}$ & count \\ \hline
  4 & 15188 \\
  0 & 6322 \\
  3 & 2038 \\
  7 & 564 \\
  6 & 522 \\
  8 & 510 \\
  5 & 462 \\
  1 & 428 \\
  9 & 425 \\
  10 & 313 \\
  11 & 145 \\
  12 & 63 \\
  2 & 25 \\
  13 & 10 
\end{tabular}
\begin{tabular}{r|l}
\multicolumn{2}{c}{$n=9$\quad$m=19$}\\
 $\mathcal{G}$ & count \\ \hline
  2 & 20654 \\
  7 & 1324 \\
  6 & 1304 \\
  1 & 1167 \\
  5 & 790 \\
  9 & 719 \\
  8 & 647 \\
  10 & 429 \\
  11 & 228 \\
  0 & 193 \\
  3 & 186 \\
  12 & 120 \\
  4 & 50 \\
  13 & 25 \\
  14 & 1 
\end{tabular}
\begin{tabular}{r|l}
\multicolumn{2}{c}{$n=9$\quad$m=20$}\\
 $\mathcal{G}$ & count \\ \hline
  4 & 14787 \\
  0 & 3927 \\
  3 & 2628 \\
  1 & 556 \\
  5 & 547 \\
  8 & 507 \\
  10 & 484 \\
  7 & 467 \\
  6 & 446 \\
  9 & 412 \\
  11 & 332 \\
  12 & 178 \\
  13 & 58 \\
  2 & 14 \\
  14 & 7 
\end{tabular}
\begin{tabular}{r|l}
\multicolumn{2}{c}{$n=9$\quad$m=21$}\\
 $\mathcal{G}$ & count \\ \hline
  2 & 12952 \\
  6 & 1600 \\
  1 & 1482 \\
  7 & 943 \\
  8 & 621 \\
  9 & 609 \\
  5 & 590 \\
  10 & 518 \\
  11 & 428 \\
  12 & 321 \\
  0 & 202 \\
  13 & 140 \\
  3 & 95 \\
  4 & 35 \\
  14 & 34 
\end{tabular}
\begin{tabular}{r|l}
\multicolumn{2}{c}{$n=9$\quad$m=22$}\\
 $\mathcal{G}$ & count \\ \hline
  4 & 8507 \\
  3 & 1574 \\
  0 & 1297 \\
  5 & 541 \\
  1 & 528 \\
  10 & 410 \\
  8 & 408 \\
  7 & 328 \\
  11 & 320 \\
  6 & 306 \\
  12 & 256 \\
  9 & 253 \\
  13 & 141 \\
  14 & 71 \\
  2 & 17 \\
  15 & 14 
\end{tabular}
\begin{tabular}{r|l}
\multicolumn{2}{c}{$n=9$\quad$m=23$}\\
 $\mathcal{G}$ & count \\ \hline
  2 & 4802 \\
  1 & 1285 \\
  6 & 1068 \\
  7 & 498 \\
  10 & 372 \\
  8 & 341 \\
  5 & 301 \\
  9 & 286 \\
  11 & 261 \\
  12 & 216 \\
  0 & 136 \\
  13 & 132 \\
  3 & 61 \\
  14 & 56 \\
  15 & 17 \\
  4 & 6 \\
  16 & 4 
\end{tabular}
\begin{tabular}{r|l}
\multicolumn{2}{c}{$n=9$\quad$m=24$}\\
 $\mathcal{G}$ & count \\ \hline
  4 & 2600 \\
  3 & 698 \\
  0 & 363 \\
  5 & 324 \\
  1 & 277 \\
  11 & 257 \\
  10 & 253 \\
  8 & 242 \\
  9 & 183 \\
  7 & 173 \\
  12 & 168 \\
  13 & 127 \\
  14 & 92 \\
  6 & 84 \\
  15 & 33 \\
  2 & 9 \\
  16 & 2 
\end{tabular}
\begin{tabular}{r|l}
\multicolumn{2}{c}{$n=9$\quad$m=25$}\\
 $\mathcal{G}$ & count \\ \hline
  2 & 907 \\
  1 & 741 \\
  6 & 430 \\
  10 & 185 \\
  7 & 160 \\
  12 & 121 \\
  9 & 110 \\
  11 & 106 \\
  5 & 103 \\
  8 & 102 \\
  13 & 77 \\
  0 & 54 \\
  14 & 44 \\
  15 & 31 \\
  3 & 20 \\
  16 & 17 \\
  4 & 2 
\end{tabular}
\begin{tabular}{r|l}
\multicolumn{2}{c}{$n=9$\quad$m=26$}\\
 $\mathcal{G}$ & count \\ \hline
  4 & 481 \\
  3 & 242 \\
  0 & 145 \\
  5 & 136 \\
  11 & 99 \\
  10 & 83 \\
  9 & 81 \\
  8 & 78 \\
  1 & 68 \\
  7 & 59 \\
  13 & 31 \\
  12 & 29 \\
  14 & 28 \\
  6 & 22 \\
  15 & 15 \\
  16 & 12 \\
  2 & 9 \\
  17 & 3 
\end{tabular}
\begin{tabular}{r|l}
\multicolumn{2}{c}{$n=9$\quad$m=27$}\\
 $\mathcal{G}$ & count \\ \hline
  1 & 281 \\
  6 & 156 \\
  2 & 81 \\
  10 & 33 \\
  0 & 33 \\
  12 & 29 \\
  9 & 25 \\
  8 & 22 \\
  5 & 22 \\
  7 & 21 \\
  11 & 20 \\
  13 & 15 \\
  3 & 10 \\
  14 & 9 \\
  15 & 5 \\
  4 & 2 \\
  16 & 1 
\end{tabular}
\begin{tabular}{r|l}
\multicolumn{2}{c}{$n=9$\quad$m=28$}\\
 $\mathcal{G}$ & count \\ \hline
  3 & 90 \\
  4 & 81 \\
  5 & 44 \\
  0 & 41 \\
  8 & 19 \\
  11 & 14 \\
  7 & 10 \\
  10 & 9 \\
  1 & 9 \\
  9 & 7 \\
  2 & 4 \\
  14 & 4 \\
  12 & 4 \\
  15 & 2 \\
  13 & 2 \\
  6 & 1 \\
  16 & 1 
\end{tabular}
\begin{tabular}{r|l}
\multicolumn{2}{c}{$n=9$\quad$m=29$}\\
 $\mathcal{G}$ & count \\ \hline
  1 & 62 \\
  6 & 50 \\
  5 & 9 \\
  2 & 7 \\
  0 & 5 \\
  4 & 4 \\
  3 & 3 \\
  10 & 3 \\
  9 & 1 \\
  8 & 1 \\
  13 & 1 \\
  11 & 1 
\end{tabular}
\begin{tabular}{r|l}
\multicolumn{2}{c}{$n=9$\quad$m=30$}\\
 $\mathcal{G}$ & count \\ \hline
  3 & 27 \\
  0 & 16 \\
  7 & 6 \\
  5 & 5 \\
  2 & 4 \\
  4 & 2 \\
  9 & 1 \\
  8 & 1 \\
  10 & 1 
\end{tabular}
\begin{tabular}{r|l}
\multicolumn{2}{c}{$n=9$\quad$m=31$}\\
 $\mathcal{G}$ & count \\ \hline
  6 & 10 \\
  1 & 8 \\
  5 & 3 \\
  0 & 3 \\
  4 & 1 
\end{tabular}
\begin{tabular}{r|l}
\multicolumn{2}{c}{$n=9$\quad$m=32$}\\
 $\mathcal{G}$ & count \\ \hline
  3 & 5 \\
  0 & 5 \\
  2 & 1 
\end{tabular}
\begin{tabular}{r|l}
\multicolumn{2}{c}{$n=9$\quad$m=33$}\\
 $\mathcal{G}$ & count \\ \hline
  5 & 4 \\
  0 & 1 
\end{tabular}
\begin{tabular}{r|l}
\multicolumn{2}{c}{$n=9$\quad$m=34$}\\
 $\mathcal{G}$ & count \\ \hline
  3 & 1 \\
  0 & 1 
\end{tabular}
\begin{tabular}{r|l}
\multicolumn{2}{c}{$n=9$\quad$m=35$}\\
 $\mathcal{G}$ & count \\ \hline
  4 & 1 
\end{tabular}
\begin{tabular}{r|l}
\multicolumn{2}{c}{$n=9$\quad$m=36$}\\
 $\mathcal{G}$ & count \\ \hline
  0 & 1 
\end{tabular}

\subsection{10 vertices}

\begin{tabular}{r|l}
\multicolumn{2}{c}{$n=10$\quad$m=10$}\\
 $\mathcal{G}$ & count \\ \hline
  0 & 1 
\end{tabular}
\begin{tabular}{r|l}
\multicolumn{2}{c}{$n=10$\quad$m=11$}\\
 $\mathcal{G}$ & count \\ \hline
  1 & 8 \\
  2 & 1 
\end{tabular}
\begin{tabular}{r|l}
\multicolumn{2}{c}{$n=10$\quad$m=12$}\\
 $\mathcal{G}$ & count \\ \hline
  0 & 84 \\
  3 & 31 \\
  4 & 6 
\end{tabular}
\begin{tabular}{r|l}
\multicolumn{2}{c}{$n=10$\quad$m=13$}\\
 $\mathcal{G}$ & count \\ \hline
  1 & 493 \\
  2 & 365 \\
  5 & 135 \\
  6 & 28 \\
  4 & 11 \\
  0 & 2 
\end{tabular}
\begin{tabular}{r|l}
\multicolumn{2}{c}{$n=10$\quad$m=14$}\\
 $\mathcal{G}$ & count \\ \hline
  0 & 2331 \\
  3 & 1902 \\
  4 & 957 \\
  7 & 273 \\
  6 & 132 \\
  5 & 122 \\
  1 & 87 \\
  8 & 52 \\
  9 & 25 \\
  2 & 13 \\
  10 & 4 
\end{tabular}
\begin{tabular}{r|l}
\multicolumn{2}{c}{$n=10$\quad$m=15$}\\
 $\mathcal{G}$ & count \\ \hline
  1 & 9615 \\
  5 & 4776 \\
  2 & 3222 \\
  6 & 2183 \\
  8 & 1311 \\
  7 & 754 \\
  4 & 531 \\
  9 & 351 \\
  10 & 264 \\
  0 & 171 \\
  3 & 112 \\
  11 & 70 \\
  12 & 10 
\end{tabular}
\begin{tabular}{r|l}
\multicolumn{2}{c}{$n=10$\quad$m=16$}\\
 $\mathcal{G}$ & count \\ \hline
  3 & 48854 \\
  0 & 8397 \\
  4 & 2818 \\
  7 & 2172 \\
  6 & 1784 \\
  9 & 1426 \\
  8 & 1261 \\
  1 & 648 \\
  10 & 632 \\
  5 & 527 \\
  11 & 325 \\
  2 & 229 \\
  12 & 93 \\
  13 & 3 
\end{tabular}
\begin{tabular}{r|l}
\multicolumn{2}{c}{$n=10$\quad$m=17$}\\
 $\mathcal{G}$ & count \\ \hline
  1 & 99738 \\
  5 & 29910 \\
  6 & 8095 \\
  2 & 4757 \\
  0 & 4263 \\
  7 & 4179 \\
  8 & 3755 \\
  4 & 3710 \\
  9 & 1719 \\
  10 & 1579 \\
  11 & 565 \\
  12 & 226 \\
  3 & 70 \\
  13 & 27 
\end{tabular}
\begin{tabular}{r|l}
\multicolumn{2}{c}{$n=10$\quad$m=18$}\\
 $\mathcal{G}$ & count \\ \hline
  3 & 233123 \\
  0 & 29690 \\
  6 & 16830 \\
  7 & 11395 \\
  8 & 8070 \\
  9 & 5677 \\
  4 & 4450 \\
  10 & 2876 \\
  5 & 1846 \\
  11 & 1304 \\
  2 & 1030 \\
  1 & 490 \\
  12 & 476 \\
  13 & 105 \\
  14 & 2 
\end{tabular}
\begin{tabular}{r|l}
\multicolumn{2}{c}{$n=10$\quad$m=19$}\\
 $\mathcal{G}$ & count \\ \hline
  1 & 370853 \\
  5 & 87675 \\
  0 & 18038 \\
  6 & 12653 \\
  7 & 9044 \\
  4 & 8029 \\
  8 & 7747 \\
  2 & 5099 \\
  9 & 5028 \\
  10 & 3806 \\
  11 & 1515 \\
  12 & 607 \\
  13 & 103 \\
  3 & 101 \\
  14 & 10 
\end{tabular}
\begin{tabular}{r|l}
\multicolumn{2}{c}{$n=10$\quad$m=20$}\\
 $\mathcal{G}$ & count \\ \hline
  3 & 497386 \\
  6 & 89080 \\
  0 & 85335 \\
  7 & 35922 \\
  8 & 24484 \\
  9 & 13172 \\
  5 & 9740 \\
  10 & 7185 \\
  4 & 4834 \\
  11 & 3236 \\
  2 & 2620 \\
  12 & 1215 \\
  1 & 408 \\
  13 & 231 \\
  14 & 28 
\end{tabular}
\begin{tabular}{r|l}
\multicolumn{2}{c}{$n=10$\quad$m=21$}\\
 $\mathcal{G}$ & count \\ \hline
  1 & 681964 \\
  5 & 196118 \\
  0 & 30168 \\
  8 & 17220 \\
  7 & 15294 \\
  9 & 14919 \\
  6 & 11927 \\
  10 & 11288 \\
  4 & 8935 \\
  11 & 6667 \\
  2 & 5792 \\
  12 & 3046 \\
  13 & 827 \\
  3 & 182 \\
  14 & 165 \\
  15 & 5 \\
  16 & 2 
\end{tabular}
\begin{tabular}{r|l}
\multicolumn{2}{c}{$n=10$\quad$m=22$}\\
 $\mathcal{G}$ & count \\ \hline
  3 & 625428 \\
  0 & 186531 \\
  6 & 174645 \\
  7 & 60915 \\
  8 & 40747 \\
  9 & 22049 \\
  5 & 15571 \\
  10 & 14536 \\
  11 & 9439 \\
  12 & 5467 \\
  4 & 4511 \\
  2 & 4315 \\
  13 & 2138 \\
  14 & 520 \\
  1 & 246 \\
  15 & 58 
\end{tabular}
\begin{tabular}{r|l}
\multicolumn{2}{c}{$n=10$\quad$m=23$}\\
 $\mathcal{G}$ & count \\ \hline
  1 & 724676 \\
  5 & 312881 \\
  9 & 31990 \\
  8 & 30896 \\
  7 & 23683 \\
  0 & 23059 \\
  10 & 21243 \\
  11 & 14501 \\
  12 & 9004 \\
  2 & 8889 \\
  6 & 8679 \\
  4 & 7312 \\
  13 & 4810 \\
  14 & 2071 \\
  3 & 418 \\
  15 & 301 \\
  16 & 17 
\end{tabular}
\begin{tabular}{r|l}
\multicolumn{2}{c}{$n=10$\quad$m=24$}\\
 $\mathcal{G}$ & count \\ \hline
  3 & 505007 \\
  0 & 321373 \\
  6 & 134849 \\
  7 & 71869 \\
  8 & 40368 \\
  9 & 23980 \\
  10 & 17980 \\
  11 & 14795 \\
  12 & 9981 \\
  13 & 6396 \\
  2 & 5744 \\
  5 & 5686 \\
  14 & 3282 \\
  4 & 3186 \\
  15 & 1110 \\
  1 & 391 \\
  16 & 155 \\
  17 & 1 
\end{tabular}
\begin{tabular}{r|l}
\multicolumn{2}{c}{$n=10$\quad$m=25$}\\
 $\mathcal{G}$ & count \\ \hline
  1 & 439022 \\
  5 & 348904 \\
  9 & 44514 \\
  8 & 42521 \\
  7 & 27873 \\
  10 & 24627 \\
  11 & 18494 \\
  12 & 13398 \\
  2 & 12039 \\
  13 & 9345 \\
  6 & 9253 \\
  0 & 7577 \\
  4 & 5511 \\
  14 & 5280 \\
  15 & 2205 \\
  3 & 1020 \\
  16 & 563 \\
  17 & 41 
\end{tabular}
\begin{tabular}{r|l}
\multicolumn{2}{c}{$n=10$\quad$m=26$}\\
 $\mathcal{G}$ & count \\ \hline
  0 & 362688 \\
  3 & 247515 \\
  6 & 52019 \\
  7 & 46894 \\
  8 & 21152 \\
  9 & 13210 \\
  10 & 12002 \\
  11 & 11465 \\
  12 & 9514 \\
  13 & 7473 \\
  2 & 6628 \\
  14 & 5344 \\
  15 & 2861 \\
  4 & 2046 \\
  16 & 1106 \\
  5 & 633 \\
  1 & 400 \\
  17 & 179 \\
  18 & 9 
\end{tabular}
\begin{tabular}{r|l}
\multicolumn{2}{c}{$n=10$\quad$m=27$}\\
 $\mathcal{G}$ & count \\ \hline
  5 & 236721 \\
  1 & 151352 \\
  8 & 38745 \\
  9 & 33986 \\
  7 & 26005 \\
  6 & 16832 \\
  10 & 15490 \\
  2 & 12935 \\
  11 & 12502 \\
  12 & 10734 \\
  13 & 8353 \\
  14 & 5932 \\
  4 & 5320 \\
  15 & 3660 \\
  3 & 2501 \\
  16 & 1568 \\
  0 & 850 \\
  17 & 435 \\
  18 & 36 \\
  19 & 1 
\end{tabular}
\begin{tabular}{r|l}
\multicolumn{2}{c}{$n=10$\quad$m=28$}\\
 $\mathcal{G}$ & count \\ \hline
  0 & 220461 \\
  3 & 86579 \\
  7 & 18998 \\
  6 & 13633 \\
  11 & 6181 \\
  8 & 6123 \\
  12 & 5886 \\
  10 & 5514 \\
  13 & 5178 \\
  9 & 4671 \\
  14 & 4618 \\
  2 & 3770 \\
  15 & 3422 \\
  16 & 1953 \\
  4 & 1395 \\
  17 & 791 \\
  1 & 382 \\
  18 & 143 \\
  5 & 79 \\
  19 & 2 
\end{tabular}
\begin{tabular}{r|l}
\multicolumn{2}{c}{$n=10$\quad$m=29$}\\
 $\mathcal{G}$ & count \\ \hline
  5 & 97097 \\
  1 & 30035 \\
  8 & 21255 \\
  7 & 16494 \\
  9 & 14181 \\
  2 & 12660 \\
  6 & 12413 \\
  10 & 5504 \\
  4 & 5496 \\
  12 & 4820 \\
  11 & 4637 \\
  13 & 3839 \\
  14 & 3169 \\
  3 & 2564 \\
  15 & 2545 \\
  16 & 1493 \\
  17 & 827 \\
  18 & 223 \\
  0 & 81 \\
  19 & 29 
\end{tabular}
\begin{tabular}{r|l}
\multicolumn{2}{c}{$n=10$\quad$m=30$}\\
 $\mathcal{G}$ & count \\ \hline
  0 & 72296 \\
  3 & 26122 \\
  7 & 6976 \\
  6 & 3494 \\
  12 & 3193 \\
  11 & 3079 \\
  13 & 3011 \\
  10 & 2927 \\
  14 & 2713 \\
  15 & 2253 \\
  9 & 1957 \\
  2 & 1864 \\
  8 & 1757 \\
  16 & 1420 \\
  17 & 941 \\
  4 & 792 \\
  18 & 368 \\
  1 & 273 \\
  5 & 66 \\
  19 & 63 \\
  20 & 6 
\end{tabular}
\begin{tabular}{r|l}
\multicolumn{2}{c}{$n=10$\quad$m=31$}\\
 $\mathcal{G}$ & count \\ \hline
  5 & 24342 \\
  2 & 9604 \\
  7 & 8050 \\
  8 & 6756 \\
  1 & 4277 \\
  9 & 3669 \\
  4 & 3272 \\
  6 & 2526 \\
  12 & 1262 \\
  10 & 1161 \\
  13 & 1159 \\
  11 & 1020 \\
  14 & 923 \\
  15 & 797 \\
  16 & 707 \\
  3 & 565 \\
  17 & 507 \\
  18 & 267 \\
  19 & 98 \\
  0 & 30 \\
  20 & 7 
\end{tabular}
\begin{tabular}{r|l}
\multicolumn{2}{c}{$n=10$\quad$m=32$}\\
 $\mathcal{G}$ & count \\ \hline
  0 & 16722 \\
  3 & 4879 \\
  7 & 2215 \\
  10 & 1458 \\
  11 & 1211 \\
  12 & 1143 \\
  9 & 1115 \\
  13 & 1065 \\
  14 & 828 \\
  6 & 698 \\
  8 & 629 \\
  15 & 623 \\
  2 & 499 \\
  16 & 413 \\
  4 & 348 \\
  17 & 301 \\
  1 & 154 \\
  18 & 120 \\
  5 & 62 \\
  19 & 50 \\
  20 & 13 \\
  21 & 2 
\end{tabular}
\begin{tabular}{r|l}
\multicolumn{2}{c}{$n=10$\quad$m=33$}\\
 $\mathcal{G}$ & count \\ \hline
  2 & 5453 \\
  5 & 3350 \\
  4 & 1388 \\
  7 & 1329 \\
  8 & 1226 \\
  1 & 977 \\
  9 & 482 \\
  12 & 200 \\
  6 & 187 \\
  13 & 184 \\
  14 & 169 \\
  11 & 168 \\
  15 & 164 \\
  16 & 119 \\
  10 & 118 \\
  17 & 88 \\
  3 & 61 \\
  18 & 42 \\
  19 & 10 \\
  0 & 10 \\
  20 & 3 \\
  21 & 1 
\end{tabular}
\begin{tabular}{r|l}
\multicolumn{2}{c}{$n=10$\quad$m=34$}\\
 $\mathcal{G}$ & count \\ \hline
  0 & 2360 \\
  3 & 735 \\
  9 & 700 \\
  7 & 654 \\
  10 & 521 \\
  11 & 348 \\
  12 & 337 \\
  8 & 287 \\
  13 & 186 \\
  6 & 149 \\
  4 & 129 \\
  14 & 80 \\
  1 & 70 \\
  15 & 54 \\
  5 & 51 \\
  16 & 32 \\
  2 & 16 \\
  17 & 16 \\
  18 & 10 \\
  19 & 7 \\
  20 & 1 
\end{tabular}
\begin{tabular}{r|l}
\multicolumn{2}{c}{$n=10$\quad$m=35$}\\
 $\mathcal{G}$ & count \\ \hline
  2 & 1633 \\
  1 & 277 \\
  4 & 274 \\
  5 & 216 \\
  8 & 100 \\
  7 & 43 \\
  11 & 39 \\
  13 & 28 \\
  9 & 26 \\
  12 & 23 \\
  15 & 22 \\
  14 & 21 \\
  10 & 15 \\
  16 & 13 \\
  3 & 10 \\
  6 & 9 \\
  0 & 6 \\
  18 & 3 \\
  17 & 3 \\
  20 & 1 \\
  19 & 1 
\end{tabular}
\begin{tabular}{r|l}
\multicolumn{2}{c}{$n=10$\quad$m=36$}\\
 $\mathcal{G}$ & count \\ \hline
  0 & 231 \\
  7 & 210 \\
  9 & 151 \\
  3 & 130 \\
  8 & 111 \\
  4 & 96 \\
  5 & 43 \\
  10 & 38 \\
  1 & 24 \\
  6 & 23 \\
  11 & 22 \\
  12 & 15 \\
  2 & 2 \\
  14 & 2 \\
  16 & 1 \\
  13 & 1 
\end{tabular}
\begin{tabular}{r|l}
\multicolumn{2}{c}{$n=10$\quad$m=37$}\\
 $\mathcal{G}$ & count \\ \hline
  2 & 270 \\
  1 & 85 \\
  8 & 15 \\
  4 & 14 \\
  11 & 8 \\
  5 & 7 \\
  10 & 7 \\
  9 & 5 \\
  12 & 5 \\
  7 & 3 \\
  13 & 3 \\
  3 & 2 \\
  6 & 1 \\
  15 & 1 \\
  0 & 1 
\end{tabular}
\begin{tabular}{r|l}
\multicolumn{2}{c}{$n=10$\quad$m=38$}\\
 $\mathcal{G}$ & count \\ \hline
  4 & 73 \\
  7 & 35 \\
  3 & 20 \\
  0 & 19 \\
  6 & 5 \\
  5 & 5 \\
  8 & 4 \\
  1 & 3 \\
  9 & 1 
\end{tabular}
\begin{tabular}{r|l}
\multicolumn{2}{c}{$n=10$\quad$m=39$}\\
 $\mathcal{G}$ & count \\ \hline
  1 & 35 \\
  2 & 28 \\
  0 & 2 \\
  8 & 1 
\end{tabular}
\begin{tabular}{r|l}
\multicolumn{2}{c}{$n=10$\quad$m=40$}\\
 $\mathcal{G}$ & count \\ \hline
  4 & 10 \\
  7 & 7 \\
  3 & 3 \\
  6 & 2 \\
  2 & 2 \\
  0 & 2 
\end{tabular}
\begin{tabular}{r|l}
\multicolumn{2}{c}{$n=10$\quad$m=41$}\\
 $\mathcal{G}$ & count \\ \hline
  1 & 10 \\
  0 & 1 
\end{tabular}
\begin{tabular}{r|l}
\multicolumn{2}{c}{$n=10$\quad$m=42$}\\
 $\mathcal{G}$ & count \\ \hline
  2 & 5 
\end{tabular}
\begin{tabular}{r|l}
\multicolumn{2}{c}{$n=10$\quad$m=43$}\\
 $\mathcal{G}$ & count \\ \hline
  1 & 2 
\end{tabular}
\begin{tabular}{r|l}
\multicolumn{2}{c}{$n=10$\quad$m=44$}\\
 $\mathcal{G}$ & count \\ \hline
  2 & 1 
\end{tabular}
\begin{tabular}{r|l}
\multicolumn{2}{c}{$n=10$\quad$m=45$}\\
 $\mathcal{G}$ & count \\ \hline
  1 & 1 
\end{tabular}

\subsection{11 vertices}

\begin{tabular}{r|l}
\multicolumn{2}{c}{$n=11$\quad$m=11$}\\
 $\mathcal{G}$ & count \\ \hline
  1 & 1 
\end{tabular}
\begin{tabular}{r|l}
\multicolumn{2}{c}{$n=11$\quad$m=12$}\\
 $\mathcal{G}$ & count \\ \hline
  0 & 10 \\
  3 & 1 
\end{tabular}
\begin{tabular}{r|l}
\multicolumn{2}{c}{$n=11$\quad$m=13$}\\
 $\mathcal{G}$ & count \\ \hline
  1 & 136 \\
  2 & 47 \\
  5 & 6 
\end{tabular}
\begin{tabular}{r|l}
\multicolumn{2}{c}{$n=11$\quad$m=14$}\\
 $\mathcal{G}$ & count \\ \hline
  0 & 1187 \\
  3 & 759 \\
  4 & 260 \\
  7 & 30 \\
  5 & 4 \\
  1 & 2 
\end{tabular}
\begin{tabular}{r|l}
\multicolumn{2}{c}{$n=11$\quad$m=15$}\\
 $\mathcal{G}$ & count \\ \hline
  1 & 7135 \\
  2 & 5779 \\
  5 & 2729 \\
  6 & 1029 \\
  4 & 255 \\
  8 & 213 \\
  7 & 146 \\
  0 & 107 \\
  3 & 60 \\
  9 & 36 \\
  11 & 1 \\
  10 & 1 
\end{tabular}
\begin{tabular}{r|l}
\multicolumn{2}{c}{$n=11$\quad$m=16$}\\
 $\mathcal{G}$ & count \\ \hline
  0 & 31432 \\
  3 & 28437 \\
  4 & 14933 \\
  7 & 7564 \\
  6 & 2903 \\
  9 & 2849 \\
  8 & 2465 \\
  5 & 1432 \\
  2 & 898 \\
  10 & 825 \\
  1 & 501 \\
  11 & 213 \\
  12 & 31 \\
  13 & 1 
\end{tabular}
\begin{tabular}{r|l}
\multicolumn{2}{c}{$n=11$\quad$m=17$}\\
 $\mathcal{G}$ & count \\ \hline
  2 & 209340 \\
  1 & 71398 \\
  5 & 28544 \\
  6 & 19357 \\
  8 & 12371 \\
  4 & 8320 \\
  7 & 7869 \\
  10 & 7457 \\
  9 & 6178 \\
  0 & 4876 \\
  11 & 2952 \\
  12 & 1325 \\
  3 & 420 \\
  13 & 105 \\
  14 & 16 
\end{tabular}
\begin{tabular}{r|l}
\multicolumn{2}{c}{$n=11$\quad$m=18$}\\
 $\mathcal{G}$ & count \\ \hline
  0 & 594015 \\
  4 & 308192 \\
  3 & 69801 \\
  7 & 68460 \\
  6 & 48000 \\
  8 & 30106 \\
  5 & 27693 \\
  9 & 27612 \\
  10 & 14439 \\
  11 & 9199 \\
  1 & 6624 \\
  12 & 4581 \\
  2 & 1863 \\
  13 & 1285 \\
  14 & 131 \\
  15 & 1 
\end{tabular}
\begin{tabular}{r|l}
\multicolumn{2}{c}{$n=11$\quad$m=19$}\\
 $\mathcal{G}$ & count \\ \hline
  2 & 2585833 \\
  1 & 141088 \\
  5 & 102880 \\
  6 & 90065 \\
  8 & 72066 \\
  7 & 58706 \\
  9 & 43958 \\
  10 & 35752 \\
  4 & 18204 \\
  11 & 17793 \\
  12 & 10061 \\
  0 & 9984 \\
  3 & 3465 \\
  13 & 3426 \\
  14 & 985 \\
  15 & 28 
\end{tabular}
\begin{tabular}{r|l}
\multicolumn{2}{c}{$n=11$\quad$m=20$}\\
 $\mathcal{G}$ & count \\ \hline
  0 & 3642348 \\
  4 & 2644991 \\
  6 & 181325 \\
  7 & 175769 \\
  3 & 121315 \\
  5 & 116215 \\
  8 & 108203 \\
  9 & 79719 \\
  10 & 50630 \\
  1 & 28513 \\
  11 & 27213 \\
  12 & 12898 \\
  13 & 4350 \\
  2 & 1718 \\
  14 & 1503 \\
  15 & 309 \\
  16 & 7 
\end{tabular}
\begin{tabular}{r|l}
\multicolumn{2}{c}{$n=11$\quad$m=21$}\\
 $\mathcal{G}$ & count \\ \hline
  2 & 12681816 \\
  6 & 301071 \\
  5 & 239894 \\
  7 & 237343 \\
  8 & 198815 \\
  1 & 189095 \\
  9 & 130455 \\
  10 & 87132 \\
  11 & 43171 \\
  4 & 26559 \\
  12 & 20014 \\
  3 & 12758 \\
  0 & 9964 \\
  13 & 5771 \\
  14 & 1579 \\
  15 & 401 \\
  16 & 63 \\
  17 & 2 
\end{tabular}
\begin{tabular}{r|l}
\multicolumn{2}{c}{$n=11$\quad$m=22$}\\
 $\mathcal{G}$ & count \\ \hline
  4 & 11811367 \\
  0 & 10772547 \\
  7 & 393406 \\
  3 & 392043 \\
  6 & 370575 \\
  8 & 323946 \\
  5 & 256131 \\
  9 & 218552 \\
  10 & 144096 \\
  11 & 70863 \\
  1 & 64577 \\
  12 & 33040 \\
  13 & 9711 \\
  14 & 2306 \\
  2 & 1803 \\
  15 & 456 \\
  16 & 69 \\
  17 & 1 
\end{tabular}
\begin{tabular}{r|l}
\multicolumn{2}{c}{$n=11$\quad$m=23$}\\
 $\mathcal{G}$ & count \\ \hline
  2 & 35974508 \\
  6 & 726600 \\
  7 & 629592 \\
  8 & 414230 \\
  5 & 408609 \\
  9 & 312808 \\
  1 & 281902 \\
  10 & 205227 \\
  11 & 110593 \\
  12 & 57136 \\
  4 & 30285 \\
  3 & 28160 \\
  0 & 20197 \\
  13 & 18073 \\
  14 & 4311 \\
  15 & 486 \\
  16 & 61 \\
  17 & 4 
\end{tabular}
\begin{tabular}{r|l}
\multicolumn{2}{c}{$n=11$\quad$m=24$}\\
 $\mathcal{G}$ & count \\ \hline
  4 & 33222437 \\
  0 & 18282530 \\
  3 & 1252117 \\
  7 & 662744 \\
  8 & 642944 \\
  6 & 570069 \\
  9 & 448136 \\
  5 & 325407 \\
  10 & 319532 \\
  11 & 174262 \\
  1 & 131292 \\
  12 & 90885 \\
  13 & 33273 \\
  14 & 10242 \\
  2 & 1471 \\
  15 & 1430 \\
  16 & 126 \\
  17 & 9 
\end{tabular}
\begin{tabular}{r|l}
\multicolumn{2}{c}{$n=11$\quad$m=25$}\\
 $\mathcal{G}$ & count \\ \hline
  2 & 67154824 \\
  6 & 1397128 \\
  7 & 1230010 \\
  8 & 819186 \\
  9 & 667468 \\
  1 & 557734 \\
  5 & 528852 \\
  10 & 469963 \\
  11 & 284714 \\
  12 & 169874 \\
  13 & 72442 \\
  0 & 56796 \\
  3 & 42687 \\
  14 & 26190 \\
  4 & 23717 \\
  15 & 4948 \\
  16 & 727 \\
  17 & 26 
\end{tabular}
\begin{tabular}{r|l}
\multicolumn{2}{c}{$n=11$\quad$m=26$}\\
 $\mathcal{G}$ & count \\ \hline
  4 & 63038406 \\
  0 & 17454171 \\
  3 & 2960262 \\
  8 & 880331 \\
  7 & 788452 \\
  6 & 710817 \\
  9 & 665758 \\
  10 & 534471 \\
  11 & 335510 \\
  5 & 304381 \\
  1 & 301481 \\
  12 & 207626 \\
  13 & 107256 \\
  14 & 46556 \\
  15 & 12934 \\
  16 & 2450 \\
  2 & 1138 \\
  17 & 141 \\
  18 & 2 
\end{tabular}
\begin{tabular}{r|l}
\multicolumn{2}{c}{$n=11$\quad$m=27$}\\
 $\mathcal{G}$ & count \\ \hline
  2 & 86823105 \\
  6 & 2555602 \\
  7 & 1794764 \\
  8 & 1430024 \\
  1 & 1211677 \\
  9 & 1177935 \\
  10 & 909291 \\
  11 & 577782 \\
  5 & 522004 \\
  12 & 375352 \\
  13 & 195447 \\
  0 & 114544 \\
  14 & 96942 \\
  3 & 71123 \\
  15 & 31953 \\
  4 & 10959 \\
  16 & 8269 \\
  17 & 753 \\
  18 & 45 
\end{tabular}
\begin{tabular}{r|l}
\multicolumn{2}{c}{$n=11$\quad$m=28$}\\
 $\mathcal{G}$ & count \\ \hline
  4 & 79703160 \\
  0 & 9064209 \\
  3 & 5647841 \\
  8 & 889351 \\
  7 & 796384 \\
  9 & 756856 \\
  6 & 734744 \\
  10 & 683967 \\
  11 & 497995 \\
  1 & 496804 \\
  12 & 355499 \\
  5 & 253262 \\
  13 & 228394 \\
  14 & 130387 \\
  15 & 57487 \\
  16 & 19367 \\
  17 & 3160 \\
  2 & 894 \\
  18 & 311 \\
  19 & 3 
\end{tabular}
\begin{tabular}{r|l}
\multicolumn{2}{c}{$n=11$\quad$m=29$}\\
 $\mathcal{G}$ & count \\ \hline
  2 & 79352188 \\
  6 & 4110963 \\
  1 & 2404976 \\
  7 & 2035620 \\
  8 & 1905090 \\
  9 & 1530901 \\
  10 & 1281469 \\
  11 & 808695 \\
  12 & 534827 \\
  5 & 385921 \\
  13 & 325133 \\
  14 & 194062 \\
  0 & 140569 \\
  3 & 100554 \\
  15 & 92201 \\
  16 & 37692 \\
  17 & 8648 \\
  4 & 3197 \\
  18 & 1301 \\
  19 & 39 
\end{tabular}
\begin{tabular}{r|l}
\multicolumn{2}{c}{$n=11$\quad$m=30$}\\
 $\mathcal{G}$ & count \\ \hline
  4 & 67062913 \\
  3 & 7484052 \\
  0 & 2938840 \\
  1 & 775943 \\
  10 & 766077 \\
  7 & 745591 \\
  9 & 730037 \\
  8 & 726002 \\
  6 & 634320 \\
  11 & 624885 \\
  12 & 466357 \\
  13 & 340462 \\
  14 & 225212 \\
  5 & 211716 \\
  15 & 128753 \\
  16 & 61474 \\
  17 & 19223 \\
  18 & 3906 \\
  2 & 1966 \\
  19 & 297 \\
  20 & 2 
\end{tabular}
\begin{tabular}{r|l}
\multicolumn{2}{c}{$n=11$\quad$m=31$}\\
 $\mathcal{G}$ & count \\ \hline
  2 & 50977458 \\
  6 & 4973298 \\
  1 & 3690412 \\
  7 & 1896931 \\
  8 & 1829597 \\
  9 & 1423453 \\
  10 & 1322544 \\
  11 & 818581 \\
  12 & 539041 \\
  13 & 367048 \\
  14 & 240021 \\
  5 & 239686 \\
  15 & 141958 \\
  0 & 101367 \\
  3 & 78278 \\
  16 & 74371 \\
  17 & 27668 \\
  18 & 7130 \\
  4 & 1132 \\
  19 & 886 \\
  20 & 34 
\end{tabular}
\begin{tabular}{r|l}
\multicolumn{2}{c}{$n=11$\quad$m=32$}\\
 $\mathcal{G}$ & count \\ \hline
  4 & 38501464 \\
  3 & 6381586 \\
  1 & 932652 \\
  0 & 870672 \\
  10 & 821872 \\
  11 & 729933 \\
  9 & 720810 \\
  7 & 611220 \\
  8 & 569964 \\
  12 & 544636 \\
  6 & 427047 \\
  13 & 417659 \\
  14 & 290765 \\
  15 & 185427 \\
  5 & 177357 \\
  16 & 110851 \\
  17 & 47791 \\
  18 & 16364 \\
  2 & 4657 \\
  19 & 2811 \\
  20 & 220 \\
  21 & 3 
\end{tabular}
\begin{tabular}{r|l}
\multicolumn{2}{c}{$n=11$\quad$m=33$}\\
 $\mathcal{G}$ & count \\ \hline
  2 & 20907514 \\
  6 & 4569451 \\
  1 & 3824767 \\
  7 & 1241034 \\
  8 & 1236178 \\
  10 & 1212368 \\
  9 & 1156068 \\
  11 & 812578 \\
  12 & 576716 \\
  13 & 443773 \\
  14 & 329142 \\
  15 & 229900 \\
  5 & 183368 \\
  16 & 156314 \\
  17 & 85549 \\
  0 & 64791 \\
  18 & 38121 \\
  3 & 35515 \\
  19 & 11044 \\
  20 & 1513 \\
  4 & 1098 \\
  21 & 63 
\end{tabular}
\begin{tabular}{r|l}
\multicolumn{2}{c}{$n=11$\quad$m=34$}\\
 $\mathcal{G}$ & count \\ \hline
  4 & 15485426 \\
  3 & 3880870 \\
  1 & 592751 \\
  10 & 535750 \\
  11 & 525832 \\
  9 & 451027 \\
  12 & 424563 \\
  13 & 367264 \\
  0 & 335609 \\
  8 & 324122 \\
  14 & 308077 \\
  7 & 293167 \\
  15 & 248663 \\
  16 & 196989 \\
  6 & 146357 \\
  17 & 134352 \\
  5 & 105608 \\
  18 & 83282 \\
  19 & 36744 \\
  20 & 9810 \\
  2 & 4831 \\
  21 & 1105 \\
  22 & 45 
\end{tabular}
\begin{tabular}{r|l}
\multicolumn{2}{c}{$n=11$\quad$m=35$}\\
 $\mathcal{G}$ & count \\ \hline
  2 & 4479125 \\
  6 & 3240081 \\
  1 & 2715569 \\
  10 & 718191 \\
  9 & 675024 \\
  8 & 535344 \\
  7 & 480721 \\
  11 & 480365 \\
  12 & 350944 \\
  13 & 282359 \\
  14 & 240526 \\
  15 & 194792 \\
  16 & 158613 \\
  5 & 156270 \\
  17 & 119180 \\
  18 & 86051 \\
  19 & 51766 \\
  0 & 43594 \\
  20 & 22469 \\
  3 & 12692 \\
  21 & 5811 \\
  4 & 1309 \\
  22 & 661 \\
  23 & 12 
\end{tabular}
\begin{tabular}{r|l}
\multicolumn{2}{c}{$n=11$\quad$m=36$}\\
 $\mathcal{G}$ & count \\ \hline
  4 & 4226043 \\
  3 & 1962342 \\
  11 & 234021 \\
  10 & 218739 \\
  12 & 203854 \\
  13 & 195048 \\
  14 & 184044 \\
  9 & 177427 \\
  15 & 164044 \\
  0 & 152327 \\
  16 & 140338 \\
  8 & 139094 \\
  1 & 132405 \\
  7 & 116058 \\
  17 & 113257 \\
  18 & 87467 \\
  19 & 56874 \\
  5 & 46944 \\
  20 & 30377 \\
  6 & 16647 \\
  21 & 12094 \\
  2 & 5828 \\
  22 & 2857 \\
  23 & 267 \\
  24 & 5 
\end{tabular}
\begin{tabular}{r|l}
\multicolumn{2}{c}{$n=11$\quad$m=37$}\\
 $\mathcal{G}$ & count \\ \hline
  6 & 1566799 \\
  1 & 1423102 \\
  2 & 423002 \\
  10 & 161130 \\
  9 & 136951 \\
  8 & 113856 \\
  5 & 110205 \\
  11 & 99482 \\
  7 & 92784 \\
  12 & 78265 \\
  13 & 68886 \\
  14 & 64609 \\
  15 & 56699 \\
  16 & 49244 \\
  17 & 41837 \\
  18 & 34123 \\
  19 & 25563 \\
  0 & 24112 \\
  20 & 15521 \\
  21 & 7646 \\
  3 & 2724 \\
  22 & 2689 \\
  4 & 1374 \\
  23 & 571 \\
  24 & 54 
\end{tabular}
\begin{tabular}{r|l}
\multicolumn{2}{c}{$n=11$\quad$m=38$}\\
 $\mathcal{G}$ & count \\ \hline
  3 & 769324 \\
  4 & 755623 \\
  11 & 83640 \\
  10 & 78458 \\
  0 & 69077 \\
  12 & 67784 \\
  13 & 65266 \\
  8 & 61734 \\
  14 & 60896 \\
  9 & 60077 \\
  7 & 49766 \\
  15 & 47646 \\
  16 & 36261 \\
  17 & 25815 \\
  18 & 17572 \\
  5 & 15971 \\
  19 & 10609 \\
  20 & 5632 \\
  1 & 4328 \\
  2 & 2950 \\
  21 & 2531 \\
  6 & 972 \\
  22 & 888 \\
  23 & 231 \\
  24 & 40 \\
  25 & 2 
\end{tabular}
\begin{tabular}{r|l}
\multicolumn{2}{c}{$n=11$\quad$m=39$}\\
 $\mathcal{G}$ & count \\ \hline
  1 & 475360 \\
  6 & 387682 \\
  5 & 56853 \\
  2 & 29527 \\
  10 & 11082 \\
  0 & 9891 \\
  14 & 9606 \\
  13 & 9421 \\
  8 & 9365 \\
  12 & 8972 \\
  15 & 8870 \\
  7 & 8495 \\
  11 & 8278 \\
  16 & 8102 \\
  9 & 7109 \\
  17 & 7065 \\
  18 & 5254 \\
  19 & 3584 \\
  20 & 1875 \\
  4 & 1122 \\
  21 & 795 \\
  3 & 318 \\
  22 & 246 \\
  23 & 38 \\
  24 & 8 \\
  25 & 2 
\end{tabular}
\begin{tabular}{r|l}
\multicolumn{2}{c}{$n=11$\quad$m=40$}\\
 $\mathcal{G}$ & count \\ \hline
  3 & 198391 \\
  4 & 82083 \\
  0 & 29314 \\
  10 & 25673 \\
  7 & 20388 \\
  11 & 20198 \\
  8 & 19461 \\
  12 & 15079 \\
  9 & 13486 \\
  13 & 13403 \\
  14 & 10070 \\
  15 & 5965 \\
  5 & 4494 \\
  16 & 3559 \\
  17 & 2141 \\
  18 & 1166 \\
  2 & 769 \\
  19 & 769 \\
  20 & 377 \\
  6 & 276 \\
  21 & 189 \\
  1 & 155 \\
  22 & 50 \\
  23 & 15 \\
  24 & 1 
\end{tabular}
\begin{tabular}{r|l}
\multicolumn{2}{c}{$n=11$\quad$m=41$}\\
 $\mathcal{G}$ & count \\ \hline
  6 & 74511 \\
  1 & 66047 \\
  5 & 30115 \\
  0 & 3027 \\
  2 & 2611 \\
  14 & 2100 \\
  13 & 1751 \\
  12 & 1723 \\
  15 & 1678 \\
  16 & 1411 \\
  11 & 1305 \\
  10 & 1090 \\
  17 & 1003 \\
  8 & 996 \\
  4 & 799 \\
  9 & 792 \\
  7 & 717 \\
  18 & 560 \\
  19 & 215 \\
  20 & 107 \\
  3 & 54 \\
  21 & 42 \\
  22 & 13 \\
  23 & 4 
\end{tabular}
\begin{tabular}{r|l}
\multicolumn{2}{c}{$n=11$\quad$m=42$}\\
 $\mathcal{G}$ & count \\ \hline
  3 & 34191 \\
  0 & 10896 \\
  10 & 6421 \\
  7 & 4752 \\
  4 & 3885 \\
  8 & 3450 \\
  11 & 3029 \\
  9 & 2326 \\
  12 & 1939 \\
  13 & 1703 \\
  14 & 853 \\
  5 & 502 \\
  15 & 439 \\
  16 & 237 \\
  2 & 230 \\
  17 & 189 \\
  18 & 97 \\
  6 & 85 \\
  19 & 61 \\
  1 & 51 \\
  20 & 23 \\
  21 & 12 \\
  22 & 1 
\end{tabular}
\begin{tabular}{r|l}
\multicolumn{2}{c}{$n=11$\quad$m=43$}\\
 $\mathcal{G}$ & count \\ \hline
  5 & 11399 \\
  6 & 8345 \\
  1 & 4479 \\
  0 & 636 \\
  2 & 479 \\
  14 & 378 \\
  12 & 372 \\
  4 & 332 \\
  13 & 332 \\
  11 & 312 \\
  15 & 262 \\
  8 & 208 \\
  9 & 181 \\
  10 & 180 \\
  16 & 129 \\
  7 & 99 \\
  17 & 60 \\
  3 & 25 \\
  18 & 25 \\
  19 & 6 \\
  20 & 3 \\
  21 & 1 
\end{tabular}
\begin{tabular}{r|l}
\multicolumn{2}{c}{$n=11$\quad$m=44$}\\
 $\mathcal{G}$ & count \\ \hline
  3 & 4500 \\
  0 & 2952 \\
  10 & 606 \\
  8 & 547 \\
  9 & 473 \\
  7 & 444 \\
  4 & 241 \\
  11 & 146 \\
  12 & 55 \\
  2 & 53 \\
  13 & 51 \\
  6 & 44 \\
  5 & 32 \\
  14 & 27 \\
  1 & 22 \\
  16 & 17 \\
  15 & 14 \\
  17 & 12 \\
  18 & 5 \\
  19 & 2 \\
  20 & 1 
\end{tabular}
\begin{tabular}{r|l}
\multicolumn{2}{c}{$n=11$\quad$m=45$}\\
 $\mathcal{G}$ & count \\ \hline
  5 & 2444 \\
  6 & 462 \\
  1 & 175 \\
  4 & 129 \\
  2 & 128 \\
  0 & 104 \\
  11 & 48 \\
  8 & 32 \\
  9 & 30 \\
  12 & 26 \\
  10 & 22 \\
  7 & 18 \\
  13 & 17 \\
  3 & 15 \\
  14 & 11 
\end{tabular}
\begin{tabular}{r|l}
\multicolumn{2}{c}{$n=11$\quad$m=46$}\\
 $\mathcal{G}$ & count \\ \hline
  0 & 550 \\
  3 & 523 \\
  7 & 58 \\
  8 & 55 \\
  9 & 34 \\
  6 & 28 \\
  2 & 26 \\
  10 & 12 \\
  4 & 9 \\
  1 & 4 \\
  12 & 2 \\
  5 & 1 \\
  15 & 1 \\
  11 & 1 
\end{tabular}
\begin{tabular}{r|l}
\multicolumn{2}{c}{$n=11$\quad$m=47$}\\
 $\mathcal{G}$ & count \\ \hline
  5 & 341 \\
  4 & 64 \\
  2 & 23 \\
  0 & 18 \\
  6 & 10 \\
  1 & 7 \\
  3 & 3 \\
  7 & 1 
\end{tabular}
\begin{tabular}{r|l}
\multicolumn{2}{c}{$n=11$\quad$m=48$}\\
 $\mathcal{G}$ & count \\ \hline
  0 & 92 \\
  3 & 63 \\
  2 & 7 \\
  6 & 4 \\
  7 & 2 \\
  5 & 2 \\
  1 & 2 
\end{tabular}
\begin{tabular}{r|l}
\multicolumn{2}{c}{$n=11$\quad$m=49$}\\
 $\mathcal{G}$ & count \\ \hline
  5 & 30 \\
  4 & 27 \\
  3 & 4 \\
  0 & 4 \\
  7 & 1 \\
  1 & 1 
\end{tabular}
\begin{tabular}{r|l}
\multicolumn{2}{c}{$n=11$\quad$m=50$}\\
 $\mathcal{G}$ & count \\ \hline
  0 & 16 \\
  3 & 10 
\end{tabular}
\begin{tabular}{r|l}
\multicolumn{2}{c}{$n=11$\quad$m=51$}\\
 $\mathcal{G}$ & count \\ \hline
  4 & 8 \\
  3 & 2 \\
  2 & 1 
\end{tabular}
\begin{tabular}{r|l}
\multicolumn{2}{c}{$n=11$\quad$m=52$}\\
 $\mathcal{G}$ & count \\ \hline
  0 & 5 
\end{tabular}
\begin{tabular}{r|l}
\multicolumn{2}{c}{$n=11$\quad$m=53$}\\
 $\mathcal{G}$ & count \\ \hline
  3 & 2 
\end{tabular}
\begin{tabular}{r|l}
\multicolumn{2}{c}{$n=11$\quad$m=54$}\\
 $\mathcal{G}$ & count \\ \hline
  0 & 1 
\end{tabular}
\begin{tabular}{r|l}
\multicolumn{2}{c}{$n=11$\quad$m=55$}\\
 $\mathcal{G}$ & count \\ \hline
  2 & 1 
\end{tabular}


\end{document}